\documentclass{article}

\usepackage{arxiv}

\usepackage[utf8]{inputenc} 
\usepackage[T1]{fontenc}    
\usepackage{hyperref}       
\usepackage{url}            
\usepackage{booktabs}       
\usepackage{amsfonts}       
\usepackage{nicefrac}       
\usepackage{microtype}      
\usepackage{lipsum}
\usepackage{amsmath,amsthm}
\usepackage{amssymb}
\usepackage{fourier}
\usepackage{hyperref}
\usepackage{graphicx} 
\usepackage{dsfont,fancyhdr,nomencl,makecell}

\newtheorem{theorem}{Theorem}

\newtheorem{corollary}[theorem]{Corollary}

\newtheorem{definition}[theorem]{Definition}

\newtheorem{lemma}[theorem]{Lemma}

\newtheorem{proposition}[theorem]{Proposition}
\newtheorem{remark}[theorem]{Remark}

\newcommand{\beq} {\begin{eqnarray*}}
\newcommand{\eeq} {\end{eqnarray*}}

\title{Optimal use of auxiliary information : information geometry and empirical process}

\newcommand{\R}{\mathbb{R}}

\newcommand{\N}{\mathbb{N}}

\author{Sofiane Arradi-Alaoui\thanks{Institut de Math\'{e}matiques de Toulouse UMR 5219 ; Universit\'{e} Paul Sabatier, France. \it{Sofiane.Arradi-alaoui@math.univ-toulouse.fr}}
}

\begin{document}
\maketitle
\begin{abstract}
We incorporate into the empirical measure $\mathbb{P}_n$ the auxiliary information given by a finite collection of expectation in an optimal information geometry way. This allows to unify several methods exploiting a side information and to uniquely define an informed empirical measure $\mathbb{P}_n^I$. These methods are shown to share the same asymptotic properties. Then we study the informed empirical process $\sqrt{n} ( \mathbb{P}_n^I - P) $ subject to a true information. We establish the Glivenko-Cantelli and Donsker theorems for $\mathbb{P}_n^I$ under minimal assumptions and we quantify the asymptotic uniform variance reduction. Moreover, we prove that the informed empirical process is more concentrated than the empirical process $\sqrt{n} ( \mathbb{P}_n - P) $ for all large $n$. Finally, as an illustration of the variance reduction, we apply some of these results to the informed empirical quantiles. 
\end{abstract}

\textit{Keywords:} information geometry, side information, auxiliary information, empirical processes, empirical likelihood, uniform central limit theorem, variance reduction.

\textit{AMS  Subject Classification:} 62B11 ; 62G30 ; 62G20 ; 60F17.

\section{Introduction}\label{sec:intro}

We call auxiliary information --or side information-- any information external to an observed statistical experiment that concerns the underlying distribution. For instance, in order to improve the quality of a survey analysis, it is customary to incorporate any reliable auxiliary information available at the time of the survey such as the knowledge of one or more parameters of this population determined exactly by an exhaustive census. This principle finds its origin several centuries ago, according to \cite{bru}. Indeed, around 1740 the magistrate Jean-Baptiste François de La Michodière wanted to estimate the size of the French population by assuming that the number of marriages, births and deaths is proportional to the size of the population. He then introduced the ratio estimator which was validated by Laplace \cite{laplace}. This method is for instance detailed in \cite{sondage}. We note that it turns out to be a special case of \cite{tarima}. More recently, several authors in survey analysis and statistics have worked on the incorporation of auxiliary information after or before sampling -- see \cite{calibration}, \cite{phd_aux_sondage}.
\\

In this article, we focus on the auxiliary information which concerns the underlying distribution of the data. More precisely, we assume that the auxiliary information is given by a finite collection of expectations. In the above mentioned case of a survey the side information can be given by the expectation of a random variable on the population. Rather few systematic analysis have been carried out on how to use at best such an auxiliary information in a general setting, despite the fact that in many case studies such a methodology is used. In \cite{AB}, the raking-ratio method allows to incorporate the auxiliary information of the probabilities of one or many partitions of a set. This is not a perfect projection of the empirical measure on the set of constraints since it is a sequential procedure, incorporating each information after the other, not simultaneously. However the variance of large classes of estimators simultaneously decreases as the sample size tends to infinity faster than the number of successive partitions. The case of an independent empirical information, as for distributed data, is investigated in \cite{albertus_learning}. In \cite{tarima}, a general method is proposed to incorporate a general auxiliary information. This approach consists in minimizing the variance over a class of unbiased estimators, that is equivalent to find the smallest dispersion ellipsoid. In \cite{albertus} this method is applied to an auxiliary information brought by a finite collection of expectations. In particular, it is shown that this method is better than the raking-ratio \cite{AB} with respect to variance reduction. Moreover, in \cite{owen} a different method based on empirical likelihood is developed to also incorporate an auxiliary information given by a finite collection of expectations. The latter two methods will be compared and connected through our definition of an informed empirical measure. In \cite{zhang} the Glivenko-Cantelli and Donsker theorems are established for the empirical likelihood method -- in the special case of the class of functions $\mathcal{F} = \left\{ 1_{]- \infty, t ]}, \ t \in \R \right\}$. For more general classes, the Glivenko-Cantelli and Donsker theorems are established in \cite{albertus} under stronger assumptions by using the strong approximation results of \cite{approximation_forte_berthet}. In addition, some works have also been done on U-statistics in the presence of auxiliary information \cite{ustat} and, more recently, on informed statistical tests \cite{albertus_ztest}, \cite{albertus_chitest}.
\\

Our contribution is to define and study an informed empirical measure supported by the sample that is optimal in the sense of information geometry. We thus intend to incorporate the auxiliary information given by expectations into the empirical measure itself by defining properly the geometrical setting in which the latter can be projected. This leads to define two projection measures, the first of which satisfies the same optimization problem as that of the empirical likelihood \cite{owen}. Next we prove that it is possible to approximate these two projection measures by a common measure we call the informed empirical measure. This informed empirical measure is far easier to compute numerically than the true projections and turns out to coincide with the adaptive estimator of the measure with auxiliary information defined in \cite{albertus}. Furthermore we show that these three measures are so close that they share the same asymptotic properties in the sense of empirical process theory -- in particular the same limiting Gaussian process. This allows to unify several methods aiming to incorporate a side information, among which those mentioned above. We establish under minimal assumptions the limit theorems for the informed empirical measure indexed by a general class of functions. As a by product this extends the asymptotic result of \cite{zhang}. Moreover we derive a concentration result which shows that the informed empirical process is always more concentrated than the classical empirical process when the sample size is large enough.      
\\

The paper is structured as follows. In Section \ref{section_geometrie}, we introduce the geometrical framework and prove that the set of constraints associated to the auxiliary information has a submanifold structure. Then we show that an optimal method is to minimize the Kullback-Leibler divergence and its dual version on the set of constraints. The readers less familiar with geometrical notions can find reminders on information geometry in the book \cite{amari} -- or could refer directly to Corollary \ref{theoreme geometrie}. In Section \ref{mesure_projection}, we study these two optimization problems. More precisely, we discuss the existence and uniqueness of their solution and we give an asymptotic theorem for the Lagrange multipliers. In Section \ref{section_def_mesure}, we prove that there exists a common approximation of these solutions which allows to define the informed empirical measure -- see Definition \ref{definition_mesure_informée}. 
Then we study the informed empirical measure's weights. In Section \ref{section_resultat_asymptotique} we establish the Glivenko-Cantelli and Donsker theorems for a general class of functions $\mathcal{F}$ about the informed empirical process under minimal assumptions. We also quantify the asymptotic uniform variance reduction, which justifies the use of auxiliary information. In Section \ref{concentration_section} we derive a concentration result about the informed empirical process. Finally in Section \ref{section_quantile},  as an illustration of the variance reduction, we apply these results to the informed empirical quantile and prove that the informed estimator is asymptotically more efficient than the classical empirical quantile. As a special case, we find the same asymptotic result as in \cite{zhang_quantile}.



\section{Framework}\label{section_framework}
Let $(X_n)_{n \in \N^*}$ be a sequence of independent and identically distributed random variables (i.i.d.) defined on a probability space $\left( \Omega, \mathcal{B}, \mathbb{P} \right)$ and taking values in a measurable space $\left( \mathcal{X}, \mathcal{A} \right)$. The distribution of $X_1$ is  denoted $P : =\mathbb{P}^{X_1}$. Moreover, we assume that an auxiliary information $I$ about $P$ is available. In this article, we focus on the following particular case. We suppose that $I$ is the information brought by a finite collection of expectations with respect to $P$. More precisely, let $m \in \N^*$ and  $g  = \left(g_1, \cdots, g_m \right)^T : \mathcal{X} \to \R^m$ be an integrable function with respect to $P$.  We assume that $I$ is given by 
$$Pg := \int_{\mathcal{X}} g dP := \left(\int_{\mathcal{X}} g_1 dP, \cdots ,\int_{\mathcal{X}} g_m dP \right)^T . $$ 
We shall assume at times that $Pg = 0$ -- otherwise set $h = g - Pg$. 

We denote $\llbracket n , p \rrbracket $ the set of integers between $n$ and $p$ where $n < p$ are two integers. For each $n \in \N^*$, the random data set is denoted $\mathcal{Z}_n = \left\{X_1, \cdots, X_n \right\}$ and $\mathbb{P}_n$ the empirical measure defined by 
\begin{align*}
    \mathbb{P}_n = \frac{1}{n} \sum_{i=1}^n \delta_{X_i}.
\end{align*}
If $n \in \N^*$ is fixed then we denote $\mathcal{Z} = \mathcal{Z}_n$ and $\mathcal{P} \left(\mathcal{Z}\right)$ the set of probability measures on $\mathcal{Z}$ with positive weights. Moreover the informed empirical measure is denoted $\mathbb{P}_n^I$ and will be defined at Section \ref{section_def_mesure} in Definition \ref{definition_mesure_informée}.

Our notation for stochastic convergences are as follows. Let $(Y_n)_{n \in \N^*}$ a sequence of random variables with values in $\R^m$ with $m \in \N^*$ and let $(a_n)_{n \in \N^*} \subset \R^*$ be a real valued sequence. We write $Y_n = o_p(a_n)$ (resp. $o_{a.s.}(a_n)$) if $\left(Y_n/a_n\right)_{n \in \N^*} $ tends to $0$ in probability (resp. almost surely) as $n\to +\infty$. We write $Y_n = O_p(a_n)$ if $\left(Y_n/a_n\right)_{n \in \N^*} $ is tight. The fact that $(Y_n)_{n \in \N^*} $ converges in distribution to a random variable $Y$ as $n\to +\infty$ is denoted by $Y_n \Rightarrow Y$.

\section{A geometrical approach of auxiliary information}\label{section_geometrie}

We intend to incorporate optimally an auxiliary information into the empirical measure. The notion of optimality may be debated but the information geometry approach seems to be a coherent and interesting answer to the problem of incorporating an auxiliary information. 

Assume that an auxiliary information $I$ about $P$ is available. For $Q \in \mathcal{P} \left(\mathcal{Z}\right)$, we denote $Q \sim I$ the fact that $Q$ satisfies the auxiliary information $I$. The set of probability measures on $\mathcal{Z}$ which satisfy $I$ is defined by 
\begin{align*}
    \mathcal{P}^I \left(\mathcal{Z} \right) = \left\{ Q \in \mathcal{P}(\mathcal{Z}) , \  \  Q \sim I \right\}.
\end{align*}
As mentioned in Section \ref{section_framework}, the weights of $Q$ are positive and $I$ is given by a finite collection of expectations $Pg$. However Section \ref{sous_section_projection} remains valid for more general definitions of auxiliary information $I$. We have 
\begin{align*}
    \mathcal{P}^I \left(\mathcal{Z} \right) = \left\{ Q \in \mathcal{P}(\mathcal{Z}) , \  \  Qg= Pg \right\}.
\end{align*}
Assuming that the basic notions of information geometry are known -- such as connection, geodesic etc -- we only recall the notion of autoparallel submanifold. Let $S$ be a manifold, $M$ be a submanifold of $S$ and $\nabla$ a connection on $S$. We say that $M$ is $\nabla$-autoparallel if for every vector fields $X,Y$ on $M$, $\nabla_{X} Y$ is also a vector field on $M$.  
In this context, since $\mathcal{P} \left(\mathcal{Z}\right)$ is a finite mixture model it is also a differential manifold endowed with the dually flat structure $(\mathcal{P}(\mathcal{Z}),g_F, \nabla^{(1)},\nabla^{(-1)}) $ where $g_F$ is the Fisher metric, $\nabla^{(1)}$ is the $1$-connection and $\nabla^{(-1)}$ is the $(-1)$-connection. Moreover, the canonical divergence associated to $\mathcal{P}\left( \mathcal{Z}\right)$ is the  Kullback–Leibler divergence $KL$ -- see \cite{KL_geometry}.

In order to use the auxiliary information $I$ let project $\mathbb{P}_n \in \mathcal{P} \left(\mathcal{Z}\right)$  on $\mathcal{P}^I \left(\mathcal{Z} \right)$ in the sense of information geometry. To define properly this projection we first show that $\mathcal{P}^I \left(\mathcal{ Z} \right)$ is a submanifold, then we recall the projection theorem in information geometry and formulate our existence result.

\subsection{Submanifold structure of $\mathcal{P}^I \left(\mathcal{ Z} \right)$}

The following result states that $\mathcal{P}^I \left(\mathcal{Z}\right)$ is a $(n-2)$-dimensional submanifold in the case $m=1$.
\begin{proposition}\label{proposition_sous_variete_cas_particulier}
Assume that there exists $i \neq j$ such that $g \left(X_i \right) \neq g \left( X_j \right)$
and $Pg$ belongs to the convex hull of $\left\{ g(X_1), \cdots ,g(X_n)\right\}$.
Then $ \ \mathcal{P}^I \left( \mathcal{Z} \right)$ is a $(n-2)$-dimensional submanifold of $\mathcal{P}(\mathcal{Z})$.

\end{proposition}

\begin{proof}
Remind that $]0,1[^n$ is a submanifold of $\R^n$ as it is an open set of $\R^n$. Set 
\begin{align*}
    \mathcal{S} = \left \{ q \in ]0,1[^n, \ \sum_{i=1}^n q_i = 1\right\}.
\end{align*}
We first show that $\mathcal{S}$ is a $(n-1)$-dimensional submanifold of  $]0,1[^n$. Define the function $\theta : ]0,1[^n \rightarrow \R$ by $\theta(q) = \sum_{i=1}^n q_i.$

Observe that $\mathcal{S} = \theta^{-1}(\{1\})$. So it is enough to prove that $\theta$ is a submersion. Indeed, $\theta$ is differentiable and for all $q \in ]0,1[^n$,
\begin{align*}
    \  D\theta(q)(h) = \theta(h),\ h \in \R^n.
\end{align*}
So $D\theta(q)$ is surjective and $\theta$ is a submersion, hence $\mathcal{S}$ is a $(n-1)$-dimensional submanifold of  $]0,1[^n$. Let endow $\mathcal{S}$ with the following global chart  $\left( \mathcal{S}, \pi \right)$  
\begin{align*}
    \pi : \mathcal{S} &\to U \subset \R^{n-1} \\
          q &\mapsto (q_1,\cdots ,q_{n-1})
\end{align*}
where $U = \left\{ (q_1, \cdots ,q_{n-1}) \in \R^{n-1}, \  q_i >0, \ i \in \llbracket 1,n-1 \rrbracket  \ , \  \ \sum_{i=1}^{n-1} q_i < 1 \right\}$. Similarly endow $\mathcal{P}(\mathcal{Z})$ with the following global chart $\left( \mathcal{P}(\mathcal{Z}), \varphi \right)$  
\begin{align*}
    \varphi :\mathcal{P}(\mathcal{Z}) &\to U \subset \R^{n-1} \\
          Q &\mapsto (q_1, \cdots ,q_{n-1}).
\end{align*}
Next consider the following one to one mapping
\begin{align*}
    \psi : \mathcal{P}(\mathcal{Z}) &\to \mathcal{S} \\
            Q &\mapsto q.
\end{align*}
Notice that $\psi = \pi^{-1} \circ \varphi$. Since $\pi$ and $\varphi$ are diffeomorphims and  $\text{dim} \    \mathcal{P}(\mathcal{Z}) = \text{dim} \ \mathcal{S} = n-1$, we deduce that $\psi$ is a diffeomorphim. So $\psi^{-1}$ is also a diffeomorphism and thus an embedding. Observe that $\mathcal{P}^I \left(\mathcal{Z} \right) = \psi^{-1} \left( \mathcal{E} \right)$

where $\mathcal{E} = \left\{ q \in \mathcal{S}, \ \sum_{i=1}^n q_i g(X_i) = Pg \right\}$ is not empty because $Pg$ belongs to the convex hull of $\{ g(X_1),...,g(X_n) \}$. So it is enough to prove that  $\mathcal{E}$ is a $(n-2)$-dimensional submanifold of  $\mathcal{S}$. For this, it is sufficient to verify that
\begin{align*}
    f : \mathcal{S} &\to \R \\
        q &\mapsto  \sum_{i=1}^n q_i g(X_i)
\end{align*}
is a submersion. Let $\gamma := f \circ \pi^{-1} : U \to \R$ be defined, for all $q \in U$, by
\begin{align*}
    \gamma(q) = \sum_{i=1}^{n-1} q_i g(X_i) + \left(1 - \sum_{i=1}^{n-1} q_i \right) g(X_n).
\end{align*}
So $\gamma$ is differentiable and for all $q \in U$ 
\begin{align*}
    D\gamma(q) = (g(X_1)-g(X_n), \cdots , g(X_{n-1})-g(X_n)).
\end{align*}
Since there exists  $i \neq j $ such that $X_i \neq X_j$, we deduce that $D\gamma(q)$ is surjective. Therefore $f$ is a submersion and $\mathcal{E}$ is a $(n-2)$-dimensional submanifold. We conclude that $\mathcal{P}^I \left( \mathcal{Z} \right) = \psi^{-1}(\mathcal{E})$ is a $(n-2)$-dimensional submanifold, as $\psi^{-1}$ is an embedding. 

\end{proof}

Proposition \ref{proposition_sous_variete_cas_particulier} can be generalised for a vector of functions $g = (g_1, \cdots , g_m)^T$ with $m \in \llbracket 1,n-1 \rrbracket$. Assume that $Pg = 0$. Set 
\begin{align*}
     \forall  j \in [|1,m|], \  N_j(X) &= (g_j(X_1)-g_j(X_n), \cdots,g_j(X_{n-1})-g_j(X_n))^T ,\\
     \forall  j \in [|1,m|], \  g_j(X) &= (g_j(X_1),\cdots ,g_j(X_n))^T.
\end{align*}
Observe that
\begin{align}
    \text{dim} \left( Vect( g(X_1) - g(X_n) ,\cdots , g(X_{n-1}) - g(X_n) )\right) = \text{dim} \left( Vect( (N_j(X))_{1 \leq j \leq m} )\right). \label{mj}
\end{align}

We are ready to state the main result of Section \ref{section_geometrie}.
\begin{proposition}\label{proposition_egalite_dimension}
Assume that $0$ belongs to the convex hull of  $\left\{ g(X_1),...,g(X_n)\right\}$ and the following equality of dimensions, 
    \begin{align*}
    l:= \dim \left( Vect( g_1(X),..., g_m(X) )\right) &= \dim  \left( Vect( g(X_1) - g(X_n) ,\cdots , g(X_{n-1}) - g(X_n) )\right) \leq m.
\end{align*}
Then $ \ \mathcal{P}^I \left( \mathcal{Z} \right)$ is a $(n-1-l)$-dimensional submanifold of $\mathcal{P}(\mathcal{Z})$.
\end{proposition}
\begin{remark}
A sufficient condition to satisfy the equality of dimensions in Proposition \ref{proposition_egalite_dimension} is
\begin{align}
    \dim Vect \left( (1, \cdots,1)^T, M_1(X), \cdots , M_m(X) \right) = m+1 \label{hypothèse_utilisée}
\end{align}
with for all $k \in [|1,m|]$, $M_k(X) = g_k(X) = \left( g_k(X_1), \cdots, g_k(X_n) \right)^T$.
\end{remark}
\begin{proof}
We keep the same notations as in the previous proof.  It remains to prove that $\mathcal{E} = \left\{ q \in \mathcal{S}, \ \sum_{i=1}^n q_i g(X_i) =0  \right\}$ is a $(n-1-l)$-dimensional submanifold. For that, we need a technical lemma. 
\begin{lemma} \label{lemme_technique_algebre_lineaire}
Let $x^1,...,x^k \in \R^n$ with $k \in \llbracket 1,n-1 \rrbracket $ and $n \in \N^*$. Define, for $j \in \llbracket 1,k \rrbracket $,
\begin{align*}
 \ \tilde{x}^j = \left( x_1^j - x_n^j, \cdots , x_{n-1}^j - x_n^j \right)^T \in \R^{n-1}.
\end{align*}
Then 
\begin{align*}
    \dim \left( Vect( \tilde{x}^1, \cdots , \tilde{x}^k) \right) \leq \dim \left( Vect( x^1, \cdots , x^k) \right).
\end{align*}
Moreover if  $l :=\dim \left( Vect( \tilde{x}^1, \cdots , \tilde{x}^k) \right) = \dim \left( Vect( x^1, \cdots , x^k) \right)$

then there exists a subset $J \subset \llbracket 1,k \rrbracket $ of size $l$ such that
\begin{align*}
    \dim \left( Vect \left( (\tilde{x}^j)_{j \in J} \right) \right) = \dim \left( Vect \left( (x^j)_{j \in J} \right) \right).
\end{align*}

\end{lemma}
\begin{proof} 
If $k=\dim \left( Vect( x^1, \cdots , x^k) \right)$ then there is nothing to prove. Assume that $\dim \left( Vect( x^1, \cdots , x^k) \right) < k$. So there exists $j \in \llbracket 1,k \rrbracket$ such that  $x^j = \sum_{i \neq j} \lambda_i x^i$ with  $\lambda_1 ,..., \lambda_k \in \R$. So for all $r\in \llbracket 1,n \rrbracket $ 
\begin{align*}
    x_r^j - x_n^j &= \sum_{i \neq j } \lambda_i x_r^i - \sum_{i \neq j } \lambda_i x_n^i = \sum_{i \neq j } \lambda_i ( x_r^i - x_n^i).
\end{align*}
In others words $\tilde{x}^j = \sum_{i \neq j} \lambda_i \tilde{x}^i$. We can conclude that 
\begin{align*}
    \dim \left( Vect( \tilde{x}^1, \cdots , \tilde{x}^k) \right) \leq \dim \left( Vect( x^1, \cdots , x^k) \right).
\end{align*}
Now, assume that  $l := \dim \left( Vect( \tilde{x}^1, \cdots , \tilde{x}^k) \right) = \dim \left( Vect( x^1, \cdots , x^k) \right).$

So there exists a subset  $J \subset \llbracket 1,k \rrbracket $ of size  $l$ such that
\begin{align*}
    l = \dim \left( Vect \left( (\tilde{x}^j)_{j \in J} \right) \right).
\end{align*}
But if  $\dim \left( Vect \left( (x^j)_{j \in J} \right) \right) < l$ then there exists $r \in J$ such that $x^r = \sum_{i \in J, \ i \neq r } \lambda_i x^i$ with $\lambda_1 ,..., \lambda_k \in \R$. By the above, we deduce that $\tilde{x}^r = \sum_{i \in J, \ i \neq r } \lambda_i \tilde{x}^i$. That contradicts the fact that 
\begin{align*}
    l = \dim \left( Vect \left( (\tilde{x}^j)_{j \in J} \right) \right).
\end{align*}

\end{proof}
We can apply Lemma \ref{lemme_technique_algebre_lineaire} to  $g_1(X), \cdots , g_m(X)$. Recall that, by  (\ref{mj}),
\begin{align*} 
    \dim  \left( Vect( g(X_1) - g(X_n) ,\cdots, g(X_{n-1}) - g(X_n) )\right) = \dim  \left( Vect( (N_j(X))_{1 \leq j \leq m} )\right)
\end{align*}
where, for $j \in \llbracket 1,m \rrbracket$,
\begin{align*}
  N_j(X) &= (g_j(X_1)-g_j(X_n), \cdots,g_j(X_{n-1})-g_j(X_n))^T.
\end{align*}
By the assumption of equality of dimensions we have 
\begin{align*}
    l:= \dim  \left( Vect( g_1(X),\cdots , g_m(X) )\right) &= \dim  \left( Vect( g(X_1) - g(X_n) ,\cdots, g(X_{n-1}) - g(X_n) )\right).
\end{align*}
Hence by Lemma \ref{lemme_technique_algebre_lineaire} there exists a subset  $J \subset \llbracket 1,m \rrbracket$ of size $l$ such that 
\begin{align*}
    l = \dim \left( Vect \left( g_1(X),...,g_m(X) \right) \right) = \dim \left( Vect \left( (g_j(X))_{j \in J} \right) \right) = \dim \left( Vect \left( (N_j(X))_{j \in J} \right) \right) 
\end{align*}
Denote $\widetilde{g} = (g_j)_{j \in J} $ and notice that
\begin{align*}
    \left\{ Q \in \mathcal{P}(\mathcal{Z}) , \  \  Qg  = 0 \right\} = \left\{ Q \in \mathcal{P}(\mathcal{Z}) , \  \  Q \widetilde{g}  = 0 \right\}.
\end{align*}
So, we can select only the constraints $\widetilde{g}$. For simplicity, in what follows we denote $g = (g_j)_{j \in J}$. 
Define the function 
\begin{align*}
    f : \mathcal{S} &\to \R^l \\
        q &\mapsto  \sum_{i=1}^n q_i g(X_i).
\end{align*}
Let prove that $f$ is a submersion. Set $\gamma := f \circ \pi^{-1} : U \to \R^l$ defined for all $q \in U$ 
\begin{align*}
    \gamma(q) = \sum_{i=1}^{n-1} q_i g(X_i) + \left(1 - \sum_{i=1}^{n-1} q_i \right) g(X_n).
\end{align*}
So $\gamma$ is  differentiable and for all $q \in U$ 
\begin{align*}
    D\gamma(q) = (g(X_1)-g(X_n), \cdots , g(X_{n-1})-g(X_n)).
\end{align*}
Since 
\begin{align*}
    \text{rk}\left( D\gamma(q)  \right) &= \dim\left( Vect( g(X_1) - g(X_n) ,\cdots, g(X_{n-1}) - g(X_n) )\right) = l.
\end{align*}
We deduce that  $D\gamma(q)$ is surjective. Thus $f$ is a submersion and $\mathcal{E}$ is a $(n-1-l)$-dimensional submanifold. We can conclude that  $\mathcal{P}^I \left( \mathcal{Z} \right) = \psi^{-1}\left(\mathcal{E} \right)$ is $(n-1-l)$-dimensional submanifold because $\psi^{-1}$ is an embedding. 

\end{proof}

\subsection{Existence of a projection} \label{sous_section_projection}
Recall the projection theorem in information geometry.
\begin{theorem} \label{proj1}
Consider a dually flat manifold $S$ and denote $D$ the canonical divergence of $S$. Let $p \in S$ and $M$ be a submanifold of $S$ which is  $\nabla^*$-autoparallel. Then a necessary and sufficient condition for a point $q \in M$ to satisfy
\begin{align}
    D(p ||q) = \min_{r \in M} D(p||r)
\end{align}
is that the $\nabla$-geodesic connecting  $p$ to $q$ is orthogonal to $M$ in $q$. The point $q$ is called  $\nabla$-projection of p on $M$. Likewise if $M$ is $\nabla$-autoparallel then a necessary and sufficient condition for a point $q \in M$ to satisfy
\begin{align}
    D^*(p ||q) = \min_{r \in M} D^*(p||r)
\end{align}
is that the $\nabla^*$-geodesic connecting $p$ to $q$ is orthogonal to $M$ in $q$ and the point $q$ is called the $\nabla^*$-projection of p on $M$. 
\end{theorem}

Moreover, it is possible to relax the autoparallel submanifold assumption.

\vspace{0.1cm}

\begin{proposition} \label{proj2}
Assume that $S$ a dually flat manifold and denote $D$ the canonical divergence of $S$. Let $p \in S$ and $M$ be a submanifold of $S$. A necessary and sufficient condition for a point $q$ to be a stationary point  of the function $D(p || \cdot) : r \mapsto D(p||r)$  restraint to $M$ (resp. $D^*(p || \cdot) : r \mapsto D^*(p||r)$) is that the $\nabla$-geodesic (resp. $\nabla^*$-geodesic) connecting  $p$ to $q $ is orthogonal to $M$ in $q$.

\end{proposition}

We next apply Theorem \ref{proj1} and Proposition \ref{proj2} to define a projection $\mathbb{Q}_n^I$ of $\mathbb{P}_n$ on $\mathcal{P}^I(\mathcal{Z})$. 

\begin{corollary} \label{theoreme geometrie}
Assume that  $\mathcal{P}^I \left(\mathcal{Z} \right)$ is a submanifold of $\mathcal{P}(\mathcal{Z})$. Then 
\begin{itemize}
    \item[$\bullet$] If $\mathcal{P}^I \left(\mathcal{Z} \right)$ is a submanifold  $\nabla^{(-1)}$-autoparallel then $\mathbb{Q}_n^I$ is the $\nabla^{(1)}$-projection of $\ \mathbb{P}_n$ on $\mathcal{P}^I \left(\mathcal{Z} \right)$ that is 
    \begin{align*}
        \mathbb{Q}_n^I &\in arg \min_{Q \in \mathcal{P}^I(\mathcal{Z})} \ KL^*\left(\mathbb{P}_n || Q \right) = arg \min_{Q \in \mathcal{P}^I(\mathcal{Z})} \ KL\left(Q || \mathbb{P}_n \right).
    \end{align*}
    \item[$\bullet$] If $\mathcal{P}^I \left(\mathcal{Z} \right)$ is  a submanifold $\nabla^{(1)}$-autoparallel then $\mathbb{Q}_n^I$ is the $\nabla^{(-1)}$-projection of \  $\mathbb{P}_n$ on $\mathcal{P}^I \left(\mathcal{Z} \right)$ 
    \begin{align*}
        \mathbb{Q}_n^I &\in arg \min_{Q \in \mathcal{P}^I(\mathcal{Z})} \ KL\left(\mathbb{P}_n || Q \right).
    \end{align*}
    \item[$\bullet$] In the case where  $\mathcal{P}^I \left(\mathcal{Z} \right)$ is  autoparallel to neither of these two connections then  $ \ \mathbb{Q}_n^I$ is a stationary point of one of these two maps 
    \begin{align*}
        Q &\longmapsto KL\left(\mathbb{P}_n||Q \right), \\ 
        Q &\longmapsto KL^* \left(\mathbb{P}_n||Q \right).
    \end{align*}
\end{itemize}
\end{corollary}

\begin{remark}
Corollary \ref{theoreme geometrie} does not determine a unique informed empirical measure by projection and moreover it is generally not easy to check if the submanifold  $\mathcal{P}^I(\mathcal{Z})$ is autoparallel.
\end{remark}

\section{Two measure projections}\label{mesure_projection}  

Assume that $Pg$ belongs to the convex hull of $\left\{g(X_1), \cdots, g(X_n) \right\}$ and the assumption (\ref{hypothèse_utilisée}) is verified then by Proposition \ref{proposition_egalite_dimension}  $ \ \mathcal{P}^I\left( \mathcal{Z} \right)$ is a submanifold. By Theorem \ref{theoreme geometrie} we were able to define a projection of $\mathbb{P}_n$ on $\mathcal{P}^I(\mathcal{Z})$. The next step is to study these two optimization problems  
\begin{align}
    &arg \min_{Q \in \mathcal{P}^I\left( \mathcal{Z} \right)} KL(\mathbb{P}_n || Q),  \label{premier} \\
    &arg \min_{Q \in \mathcal{P}^I\left( \mathcal{Z} \right)} KL(Q || \mathbb{P}_n).  \label{second}
\end{align}
Remark that 
\begin{align*}
     KL(\mathbb{P}_n ||Q) &= \sum_{i=1}^n  \frac{1}{n} \log\left( \frac{1}{nq_i} \right) = -\log(n) - \frac{1}{n} \sum_{i=1}^n \log q_i, \\
     KL(Q || \mathbb{P}_n) &= \sum_{i=1}^n q_i \log (nq_i) = \log n + \sum_{i=1}^n q_i \log q_i .                 
\end{align*}
Therefore 
\begin{align*}
    arg \min_{Q \in \mathcal{P}^I} KL(\mathbb{P}_n || Q) &= arg \max_{Q \in \mathcal{P}^I\left( \mathcal{Z} \right)} \sum_{i=1}^n \log q_i = arg \max_{Q \in \mathcal{P}^I\left( \mathcal{Z} \right)} \prod_{i=1}^n q_i, \\
    arg \min_{Q \in \mathcal{P}^I\left( \mathcal{Z} \right)} KL(Q || \mathbb{P}_n) &=  arg \min_{Q \in \mathcal{P}^I\left( \mathcal{Z} \right)} \sum_{i=1}^n q_i \log q_i.
\end{align*}

\subsection{First optimization problem}

The first optimization problem (\ref{premier}) is called empirical likelihood and has been studied by A.B. Owen -- refer to \cite{owen}, \cite{owen_book}, \cite{qin} and \cite{zhang_vraisemblance}.  More precisely, the first optimisation problem is the following 
$$ \begin{cases}
    \max_{\textbf{q}} \sum_{i=1}^n \log(q_i) \\[0,1cm]
    \sum_{i=1}^n q_i = 1, \\[0,1cm]
    \forall i \in \llbracket 1,n \rrbracket \ , \  q_i > 0, \\[0,1cm]
    n \sum_{i=1}^n q_i (g(X_i) - Pg) = 0.
\end{cases}$$
Denote $\mathcal{C}_n$ the constraints set 
\begin{align*}
    \mathcal{C}_n = \left\{ q \in [0,1]^n, \ \sum_{i=1}^n q_i = 1 , \ n\sum_{i=1}^n q_i (g(X_i) - Pg) = 0 \ \right\}.
\end{align*}
Moreover denote also 
\begin{align*}
    \ M_k(X) &= (g_k(X_1), \cdots , g_k(X_n))^T, \ k \in \llbracket 1,m \rrbracket ,\\
  f(q) &= \sum_{i=1}^n \log q_i, \ q \in ]0,1[^n.
\end{align*}
The following theorem ensures that there is a unique solution to this problem.
\begin{theorem} \label{theoreme_premier_opt}
Assume that  $Pg$ belongs to the convex hull of  $\left\{g(X_1), \cdots, g(X_n) \right\}$ and 
\begin{align}
    \dim Vect\left( (1, \cdots,1)^T, M_1(X), \cdots , M_m(X) \right) = m+1 \label{dimension_opt_1}.
\end{align}
Then the first optimization problem has a unique solution $q^* = (q_1^*, \cdots, q_n^*) \in \mathcal{C}_n$. Moreover for all  $i \in \llbracket 1,n \rrbracket$  
\begin{align*}
    q_i^* &> 0, \\
    q_i^* &= \frac{1}{n} \frac{1}{1 + \langle \lambda^* , g(X_i) - Pg \rangle}
\end{align*}
for a unique $\lambda^* \in \bigcap_{i=1}^n \left\{ \lambda \in \R^m, \ 1 + \langle \lambda , g(X_i) - Pg \rangle \geq \frac{1}{n}  \right\} $. 
\end{theorem}
\begin{remark}
The assumption (\ref{dimension_opt_1}) is not restrictive since if it is not verified that means that some constraints are redundant. 
\end{remark}

Assume that $Pg = 0$ and write $\hat{\lambda}_n = \lambda^*$ to state an asymptotic result.

\begin{theorem} \label{theoreme asymptotique premier problème}
Assume that  $\mathbb{E}\Vert g(X) \Vert^2 < +\infty$  and $\Sigma = Pgg^T$ is positive definite. Moreover, we suppose that the assumptions of the theorem \ref{theoreme_premier_opt} are satisfied. Set $\Sigma_n = P_n g g^T$.  Then we have
\begin{itemize}
    \item[$\bullet$] $\max_{1\leq i \leq n} |< \hat{\lambda}_n,g(X_i)> | = o_p(1)$.
    \item[$\bullet$] $\hat{\lambda}_n = \Sigma_n^{-1} \mathbb{P}_n g + o_p\left(\frac{1}{\sqrt{n}} \right) $.
    \item[$\bullet$] $\sqrt{n}\hat{\lambda}_n \Rightarrow \mathcal{N}(0,\Sigma^{-1})$.
\end{itemize}
\end{theorem}
\begin{remark}
We can replace $\Sigma_n$ by the empirical variance $Var_n(g)$ of  $g$ since  $$ \Sigma_n = Var_n(g)  +  (\mathbb{P}_n g )^T \mathbb{P}_n g  = Var_n(g)  + o_p\left(\frac{1}{\sqrt{n}}\right). $$ 
\end{remark}
The last two theorems can be proved by using a similar approach that in \cite{owen}.

\subsection{Second optimization problem}

Let consider the second optimization problem (\ref{second}) 
$$ \begin{cases}
    \min_{\textbf{q}} \sum_{i=1}^n q_i \log(q_i) \\[0,1cm]
    \sum_{i=1}^n q_i = 1, \\[0,1cm]
    \forall i \in \llbracket 1,n \rrbracket \ , \  q_i > 0, \\[0,1cm]
    n \sum_{i=1}^n q_i (g(X_i) - Pg) = 0.
\end{cases}$$
Set 
\begin{align*}
 \varphi(q) &= \sum_{i=1}^n q_i \log q_i,\ q \in ]0,1[^n.
\end{align*}
The following result ensures that there is a unique solution to the second optimization problem. 
\begin{theorem} \label{solution second problème}
Assume that $Pg$ belongs to the convex hull of $\left\{g(X_1), \cdots, g(X_n) \right\}$ and 
\begin{align}
    \dim Vect\left( (1, \cdots,1)^T, M_1(X), \cdots , M_m(X) \right) = m+1 \label{dimension_opt_3}.
\end{align}
Then the second optimization problem has a unique solution $q^* = (q_1^*, \cdots, q_n^*) \in \mathcal{C}_n$. Moreover for all  $i \in [|1,n|]$  
\begin{align*}
    q_i^* &= \frac{\exp \left( \langle \lambda^* , g(X_i) - Pg \rangle \right) }{\sum_{j=1}^n \exp \left( \langle \lambda^* , g(X_j) - Pg \rangle \right)} 
\end{align*}
for a unique $\lambda^* \in \R^m $. 
\end{theorem}
\begin{proof}

\textit{Existence.} Since $\mathcal{C}_n$ is compact and $\varphi$ is continuous, there exists $q^* \in \mathcal{C}_n$ such that 
\begin{align*}
    \sum_{i=1}^n q_i^* \log q_i^* = \min_{q \in \mathcal{C}_n} \sum_{i=1}^n q_i \log q_i.
\end{align*}

\textit{Uniqueness.} By the assumption  (\ref{dimension_opt_3}) and $q^*$ is a global maximum of  $\varphi$ on  $\mathcal{C}_n$, we can apply the Lagrange multiplier theorem. So there exists  $(\lambda^*,\mu) \in \R^m \times \R$ such that 
\begin{align*}
    D\varphi(q^*) = \mu(1, \cdots,1)^T + \sum_{j=1}^m \lambda^*_j \left( g_j(X_1) - Pg_j, \cdots , g_j(X_n) - Pg_j \right)^T.
\end{align*}
So for all  $i \in \llbracket 1,n \rrbracket$ 
\begin{align*}
    \log q_i^* + 1  = \mu +  \langle \lambda^* , g(X_i) - Pg \rangle.
\end{align*}
Since $\sum_{i=1}^n q^*_i =1 $, we deduce that 
\begin{align*}
    \mu &= 1 - \log \left( \sum_{i=1}^n \exp \left( \langle \lambda^* , g(X_i) - Pg \rangle \right) \right) \\
    q_i^* &= \frac{\exp \left( \langle \lambda^* , g(X_i) - Pg \rangle \right) }{\sum_{j=1}^n \exp \left( \langle \lambda^* , g(X_j) - Pg \rangle \right)}. 
\end{align*}
By injecting this expression, we get
\begin{align*}
    \varphi(q^*) &= \frac{ \sum_{i=1}^n \langle \lambda^* , g(X_i) - Pg \rangle \exp \left(  \langle \lambda^* , g(X_i) - Pg \rangle \right) }{\sum_{j=1}^n \exp \left( \langle \lambda^* , g(X_j) - Pg \rangle \right)} - \log \left( \sum_{i=1}^n \exp \left( \langle \lambda^* , g(X_i) - Pg \rangle \right) \right) \\
    &=\frac{  \langle \lambda^* , \sum_{i=1}^n (g(X_i) - Pg)  \exp \left(  \langle \lambda^* , g(X_i) - Pg \rangle \right) \rangle }{\sum_{j=1}^n \exp \left( \langle \lambda^* , g(X_j) - Pg \rangle \right)} - \log \left( \sum_{i=1}^n \exp \left( \langle \lambda^* , g(X_i) - Pg \rangle \right) \right).
\end{align*}
Since $\sum_{i=1}^n q_i^* \left( g(X_i) - Pg \right) =0 $, we have 
\begin{align*}
    \frac{   \sum_{i=1}^n (g(X_i) - Pg)  \exp \left(  \langle \lambda^* , g(X_i) - Pg \rangle \right)  }{\sum_{j=1}^n \exp \left( \langle \lambda^* , g(X_j) - Pg \rangle \right)} = 0.
\end{align*}
So 
\begin{align*}
    \varphi(q^*) = - \log \left( \sum_{i=1}^n \exp \left( \langle \lambda^* , g(X_i) - Pg \rangle \right) \right).
\end{align*}
Set the following function  $F : \R^m \to \R$ defined by  
\begin{align*}
    \forall \lambda \in \R^m, \ F(\lambda) =  - \log \left( \sum_{i=1}^n \exp \left( \langle \lambda , g(X_i) - Pg \rangle \right) \right).
\end{align*}
By convex duality, this is equivalent to maximize the function  $F$ on $\R^m$. This function is differentiable and concave. So  $\nabla F(\lambda) = 0 $ is equivalent to  $\lambda \in arg \max_{x \in \R^m} F(x)$. In our case 
\begin{align*}
    \nabla F( \lambda) = \frac{  - \sum_{i=1}^n (g(X_i) - Pg)  \exp \left(  \langle \lambda , g(X_i) - Pg \rangle \right)  }{\sum_{j=1}^n \exp \left( \langle \lambda , g(X_j) - Pg \rangle \right)}.
\end{align*}
Thus $\nabla F(\lambda^*) = 0$. It remains to prove that  $F$ is strictly concave. For all $j,k \in \llbracket 1,m \rrbracket $ 
\begin{align*}
    \frac{\partial^2 F}{\partial \lambda_j \partial \lambda_k} &= \frac{  - \sum_{i=1}^n (g_j(X_i) - Pg_j) (g_k(X_i) - Pg_k) \exp \left(  \langle \lambda , g(X_i) - Pg \rangle \right)  }{\sum_{j=1}^n \exp \left( \langle \lambda , g(X_j) - Pg \rangle \right)} \\ 
    &+  \frac{  \sum_{i=1}^n (g_j(X_i) - Pg_j)  \exp \left(  \langle \lambda , g(X_i) - Pg \rangle \right)  }{\sum_{j=1}^n \exp \left( \langle \lambda , g(X_j) - Pg \rangle \right)} \frac{   \sum_{i=1}^n (g_k(X_i) - Pg_k) \exp \left(  \langle \lambda , g(X_i) - Pg \rangle \right)  }{\sum_{j=1}^n \exp \left( \langle \lambda , g(X_j) - Pg \rangle \right)}.
\end{align*}
Denote for all $ \lambda \in \R^m $ and $i \in \llbracket 1,n \rrbracket$, 
\begin{align*}
   q_i(\lambda) &= \frac{ \exp \left(  \langle \lambda , g(X_i) - Pg \rangle \right)  }{\sum_{j=1}^n \exp \left( \langle \lambda , g(X_j) - Pg \rangle \right)}, \\
   Q_n(\lambda) &= \sum_{i=1}^n q_i(\lambda) \delta_{X_i}.
\end{align*}
Observe that for all $ \lambda \in \R^m $, $Q_n(\lambda)$ is a probability measure. The Hessian matrix of $F$ is 
\begin{align*}
    Hess_F = - Var_{Q_n(\lambda)} (g-Pg) = - Var_{Q_n(\lambda)} (g) < 0.
\end{align*}
By using assumption  (\ref{dimension_opt_3}), we deduce that the hessian matrix of  $F$ is negative definite. So $F$ is strictly concave.
\end{proof}
Assume that $Pg = 0$ and denote again $\hat{\lambda}_n := \lambda^* $ to derive the following asymptotic result. 
\begin{theorem} \label{theoreme asymptotique second problème}
Assume that $ \ \mathbb{E}\Vert g(X) \Vert^2 < +\infty$  and $\Sigma = Pgg^T$ is positive definite. We suppose that the assumptions of Theorem \ref{solution second problème} are satisfied. Set $\Sigma_n = P_n g g^T$. Then we have
\begin{itemize}
    \item[$\bullet$] $\max_{1\leq i \leq n} |< \hat{\lambda}_n,g(X_i)> | = o_p(1)$.
    \item[$\bullet$] $\hat{\lambda}_n = - \Sigma_n^{-1} \mathbb{P}_n g + o_p\left(\frac{1}{\sqrt{n}}\right) $.
    \item[$\bullet$] $\sqrt{n}\hat{\lambda}_n \Rightarrow \mathcal{N}(0,\Sigma^{-1})$.
\end{itemize}
\end{theorem}
\begin{remark}
Likewise we can replace  $\Sigma_n$ by the empirical variance of $g$.
\end{remark}
\begin{proof}
Firstly, let us prove that $\Vert \hat{\lambda}_n \Vert = O_p \left(\frac{1}{\sqrt{n}} \right)$. Denote $\hat{\lambda}_n =\rho_n \theta_n $ with $\rho_n  \geq 0$ , $\Vert\theta_n \Vert =1$. Set for all  $\lambda \in \R^m$ 
\begin{align*}
     q_i(\lambda) &= \frac{ \exp \left(  \langle \lambda , g(X_i) \rangle \right)  }{\sum_{j=1}^n \exp \left( \langle \lambda , g(X_j) \rangle \right)}, \\
     \varphi(\lambda) &=  \sum_{i=1}^n q_i(\lambda) g(X_i).
\end{align*}
By definition of  $\hat{\lambda}_n$, we have $\varphi(\hat{\lambda}_n) = 0$. So  
\begin{align*}
    0 &= \Vert \phi (\hat{\lambda}_n) \Vert = \Vert \phi(\rho_n \theta_n) \Vert  = \mid \langle \theta_n , \varphi(\rho_n \theta_n) \rangle \mid.
\end{align*}
By denoting $S_n(\lambda) = \sum_{j=1}^n \exp \left( \langle \lambda , g(X_j) \rangle \right) > 0 $, we have 
\begin{align*}
    \frac{\theta_n^T}{S_n ( \hat{\lambda}_n)} \sum_{i=1}^n  \exp \left( \rho_n \theta_n^T  g(X_i) \right) g(X_i) = 0.
\end{align*}

By using the Taylor's theorem with Lagrange remainder (first order) we have that for all $i \in [|1,n|]$, there exists $r_i$ between $0$ and $\theta_n^T g(X_i)$ such that
\begin{align*}
    \exp\left(\rho_n \theta_n^T g(X_i) \right) = 1 + \rho_n \theta_n^T g(X_i) \exp(\rho_n r_i).
\end{align*}
Denote $S^{m-1}$ the unit sphere of $\R^m$. So 
\begin{align*}
     \frac{\theta_n^T}{S_n ( \hat{\lambda}_n)} \sum_{i=1}^n  \exp \left( \rho_n \theta_n^T  g(X_i) \right) g(X_i) &= \frac{n\theta_n^T}{S_n ( \hat{\lambda}_n)} \mathbb{P}_n g + \frac{n \rho_n}{S_n ( \hat{\lambda}_n)} \frac{1}{n} \sum_{i=1}^n (\theta_n^T g(X_i))^2 \exp(\rho_n r_i) \\
     &\geq \frac{n\theta_n^T}{S_n ( \hat{\lambda}_n)} \mathbb{P}_n g + \frac{n \rho_n}{S_n ( \hat{\lambda}_n)} \frac{1}{n} \sum_{i=1}^n (\theta_n^T g(X_i))^2 1_{\theta_n^T g(X_i) \geq 0} \\
     &\geq \frac{n\theta_n^T}{S_n ( \hat{\lambda}_n)} \mathbb{P}_n g + \frac{n \rho_n}{S_n ( \hat{\lambda}_n)} \inf_{\theta \in S^{m-1}} \frac{1}{n} \sum_{i=1}^n (\theta^T g(X_i))^2 1_{\theta^T g(X_i) \geq 0}.
\end{align*}
since for all $i \in \llbracket 1,n \rrbracket$  if $\theta_n^T g(X_i) \geq 0$ then $r_i \geq 0$. Thus 
\begin{align*}
            \frac{n\theta_n^T}{S_n ( \hat{\lambda}_n)} \mathbb{P}_n g + \frac{n \rho_n}{S_n ( \hat{\lambda}_n)} \inf_{\theta \in S^{m-1}} \frac{1}{n} \sum_{i=1}^n (\theta^T g(X_i))^2 1_{\theta^T g(X_i) \geq 0} \leq 0.
\end{align*}
By multiplying by $\frac{S_n ( \hat{\lambda}_n)}{n}$  and since  $ \Vert \theta_n \Vert = 1$, we have 
\begin{align*}
    \rho_n \inf_{\theta \in S^{m-1}} \frac{1}{n} \sum_{i=1}^n (\theta^T g(X_i))^2 1_{\theta^T g(X_i) \geq 0} \leq \Vert \mathbb{P}_n g \Vert.
\end{align*}
It is enough to prove that $\inf_{\theta \in S^{m-1}} \frac{1}{n} \sum_{i=1}^n (\theta^T g(X_i))^2 1_{\theta^T g(X_i) \geq 0}$ converges in probability to a  strictly positive constant in order to deduce that  $\rho_n = O_p\left( \frac{1}{\sqrt{n}}\right)$. For that, we need a technical lemma. 

\begin{lemma} \label{lemme-glivenko}
We have 
$$\inf_{\theta \in S^{m-1}} \frac{1}{n} \sum_{i=1}^n (\theta^T g(X_i))^2 1_{\theta^T g(X_i) \geq 0} \xrightarrow[n \to  \infty]{a.s.}\inf_{\theta \in S^{m-1}} P (\theta^T g)^2 1_{\theta^T g \geq 0}  > 0.  $$
\end{lemma}
\begin{proof}
Set for all $\theta \in S^{m-1}$  
$$\gamma_n(\theta) =  \frac{1}{n} \sum_{i=1}^n (\theta^T g(X_i))^2 1_{\theta^T g(X_i) \geq 0}. $$
Firstly, let us prove that $\gamma_n$ is continuous. Endow  $S^{m-1}$ with the subspace topology of $\R^m$. Define for all $i\in [|1,n|]$  
$$\begin{array}{ccccc}
h_i & : & S^{m-1} & \to & \R \\
 & & \theta & \mapsto & \theta^T g(X_i).  \\
\end{array}$$
Set the following function $ l : \R \to \R $ defined by  $l(x) = x^2 1_{x \geq 0}$. Remark that for all  $i \in \llbracket 1,n \rrbracket$ the functions  $h_i$ and $l$ are continuous and $\gamma_n( \theta ) = \frac{1}{n} \sum_{i=1}^n l \circ h_i(\theta).$

So for all $n \in \N^*$, $\gamma_n$ is continuous on  $S^{m-1}$. Define $\psi : C^0(S^{m-1}) \rightarrow \mathbb{R}$ by  $\psi(f) = \inf_{\Vert \theta \Vert =1 } f(\theta)$. Let us prove that  $\psi$ is continuous. Indeed for all $f_1,f_2 \in C^0(S^{m-1})$
\begin{align*}
    \mid \psi(f_1) - \psi(f_2) \mid &= \mid \inf f_1 - \inf f_2 \mid \leq \Vert f_1 - f_2 \Vert_{\infty}.
\end{align*}
So $\psi$ is  1-Lipschitz. Notice that 
\begin{align*}
    \psi(\gamma_n) = \inf_{\theta \in S^{m-1}} \frac{1}{n} \sum_{i=1}^n (\theta^T g(X_i))^2 1_{\theta^T g(X_i) \geq 0}.
\end{align*}
Prove that  $\gamma_n$ converge uniformly almost surely  to  $\gamma$ where for all  $\theta \in S^{m-1}$   $ \gamma(\theta) = P (\theta^T g)^2 1_{\theta^T g \geq 0}. $
For that, remark that for all  $\theta \in S^{m-1}$ 
\begin{align*}
    | \gamma_n(\theta) - \gamma(\theta) | = | (\mathbb{P}_n - P)(\theta^T g)^2 1_{\theta^T g \geq 0} |.
\end{align*}
Define the following class of functions  
\begin{align*}
    \mathcal{F}= \left\{ h_\theta : x \mapsto \theta^T g(x))^2 1_{\theta^T g(x) \geq 0} \ , \ \theta \in S^{m-1} \right\}.
\end{align*}
Since the sphere $S^{m-1}$ is compact, for all  $x \in \mathcal{X}$ the map  $\theta \mapsto h_\theta(x)$ is continuous and  \begin{align*}
    \sup_{\theta \in S^{m-1}} |h_\theta| \leq \Vert g \Vert^2 \in L^1(P).
\end{align*}
We deduce that the class  $\mathcal{F}$ is $P$-Glivenko-Cantelli. So
$$ \gamma_n \xrightarrow[n \to \infty]{\Vert . \Vert_{\infty}, \  a.s.} \gamma. $$  
By continuity of $\psi$, we have   $ \psi(\gamma_n) \xrightarrow{a.s} \psi(\gamma).$
We conclude that $$\inf_{\theta \in S^{m-1}} \frac{1}{n} \sum_{i=1}^n (\theta^T g(X_i))^2 1_{\theta^T g(X_i) \geq 0} \xrightarrow[n \to  \infty]{a.s}\inf_{\theta \in S^{m-1}} P (\theta^T g)^2 1_{\theta^T g \geq 0}.    $$
Finally we show that  $ \inf_{\theta \in S^{m-1}} P (\theta^T g)^2 1_{\theta^T g \geq 0}  > 0$.
Since  $P \Vert g\Vert^2 < +\infty$, the map $\theta \mapsto P (\theta^T g)^2 1_{\theta^T g \geq 0}$ is continuous. By compactness of the sphere, there exists  $\theta_* \in S^{m-1}$ such that  
\begin{align*}
    \inf_{\theta \in S^{m-1}} P (\theta^T g)^2 1_{\theta^T g \geq 0} =  P (\theta_*^T g)^2 1_{\theta_*^T g \geq 0}.
\end{align*}
If $P(\theta_*^T g)^2 1_{\theta_*^T g \geq 0} = 0$ then  $P$-a.s x, $\theta_*^T g(x)  \leq 0$. In others words $\mathbb{P}\left( \theta_*^T g(X) > 0 \right) = 0.$
Remark that 
\begin{align*}
    0 = \theta_*^T Pg = P \theta_*^T g =  P (\theta_*^T g)1_{\theta_*^T g \leq 0}. 
\end{align*}
So $P$-a.s x, $\theta_*^T g(x) = 0$. Since $\Sigma$ is a matrix definite positive, we obtain a contradiction 
\begin{align*}
  0 = P (\theta_*^T g)^2 = \theta_*^T P gg^T \theta_* = \theta_*^T \Sigma \theta_* > 0.
\end{align*}
  
\end{proof}
Since  $\rho_n = \Vert \hat{\lambda}_n \Vert$, we deduce by Lemma \ref{lemme-glivenko}   
\begin{align*}
    \Vert \hat{\lambda}_n \Vert &= O_p\left( \frac{1}{\sqrt{n}} \right), \\
    \max_{1 \leq i \leq n} \mid \langle \hat{\lambda}_n, g(X_i) \rangle \mid &\leq \Vert \hat{\lambda}_n \Vert \max_{1 \leq i \leq n} \Vert g(X_i) \Vert.
\end{align*}
Prove that $\max_{1 \leq i \leq n} \Vert g(X_i) \Vert = o_p(\sqrt{n})$. For that, we will use Owen's lemma (1990, \cite{owen}).
\begin{lemma}\label{lemme_Owen}
Let $(Y_n)_{n \in \N^*}$ be a sequence i.i.d. of positive random variables and $Z_n = \max_{1 \leq i \leq n} Y_i$. If $\mathbb{E}Y^2 < + \infty$ then 
\begin{align*}
    Z_n &= o_{p}(\sqrt{n}), \\
    \frac{1}{n} \sum_{i=1}^n Y_i^3 &= o_{p}(\sqrt{n}).
\end{align*}
\end{lemma}
Since $\mathbb{E} \Vert g(X) \Vert^2 < + \infty$, we deduce that by Owen's lemma that  $\max_{1 \leq i \leq n} \Vert g(X_i) \Vert  = o_p(\sqrt{n})$. So we get the first assertion 
\begin{align*}
    \max_{1 \leq i \leq n} \mid \langle \hat{\lambda}_n, g(X_i) \rangle \mid = o_p(1).
\end{align*}
About the second assertion, apply the Taylor's theorem with Lagrange remainder (second order). More precisely for all  $i \in \llbracket 1,n \rrbracket$, there exists $s_i$ between  $0$ and $\theta_n^T g(X_i)$ such that 
\begin{align*}
    \exp\left( \rho_n \theta_n^T g(X_i) \right) &= 1 + \rho_n \theta_n^T g(X_i) + \rho_n^2 \left(\theta_n^T g(X_i) \right)^2 \exp\left( \rho_n s_i \right) \\
    &= 1 +  \hat{\lambda}_n^T g(X_i) + \left(\hat{\lambda}_n^T g(X_i) \right)^2 \exp\left( \rho_n s_i \right).
\end{align*}
Thus
\begin{align*}
    0 &= \frac{1}{S_n(\hat{\lambda}_n)} \sum_{i=1}^n \exp\left( \hat{\lambda}_n^T g(X_i) \right) g(X_i) \\
    &= \frac{1}{S_n(\hat{\lambda}_n)} \sum_{i=1}^n g(X_i) + \frac{1}{S_n(\hat{\lambda}_n)} \left( \sum_{i=1}^n g(X_i) g(X_i)^T \right)  \hat{\lambda}_n + \frac{1}{S_n(\hat{\lambda}_n)} \sum_{i=1}^n \left(\hat{\lambda}_n^T g(X_i) \right)^2 \exp\left( \rho_n s_i \right) g(X_i).
\end{align*}
By multiplying by $\frac{S_n(\hat{\lambda}_n)}{n}$, we obtain 
\begin{align*}
    &0 =  \mathbb{P}_n g  + \Sigma_n \hat{\lambda}_n + \frac{1}{n} \sum_{i=1}^n \left(\hat{\lambda}_n^T g(X_i) \right)^2 \exp\left( \rho_n s_i \right) g(X_i). 
\end{align*}
This is equivalent to 
\begin{align*}
    - \Sigma_n \hat{\lambda}_n  = \mathbb{P}_n g + \frac{1}{n} \sum_{i=1}^n \left(\hat{\lambda}_n^T g(X_i) \right)^2 \exp\left( \rho_n s_i \right) g(X_i).
\end{align*}
Finally we prove that  
\begin{align*}
    \left \Vert \frac{1}{n} \sum_{i=1}^n \left(\hat{\lambda}_n^T g(X_i) \right)^2 \exp\left( \rho_n s_i \right) g(X_i) \right \Vert = o_p\left( \frac{1}{\sqrt{n}} \right).
\end{align*}
By Owen's lemma and the first assertion, we have for all $i \in \llbracket 1,n \rrbracket$
\begin{align}
    &\exp\left( \rho_n s_i \right) \leq \exp\left( \rho_n | \theta_n^T g(X_i) |\right) \leq \exp\left( \max_{1 \leq i \leq n} | \hat{\lambda}_n^T g(X_i) |\right) = 1 + o_p(1),  \\
    & \frac{1}{n} \sum_{i=1}^n \Vert g(X_i) \Vert^3 = o_p(\sqrt{n}), \\
    &\Vert \hat{\lambda}_n \Vert^2 = O_p \left(\frac{1}{n} \right).
\end{align}
Hence
\begin{align*}
     \left \Vert \frac{1}{n} \sum_{i=1}^n \left(\hat{\lambda}_n^T g(X_i) \right)^2 \exp\left( \rho_n s_i \right) g(X_i) \right \Vert &\leq \exp\left( \max_{1 \leq i \leq n} | \hat{\lambda}_n^T g(X_i) |\right) \Vert \hat{\lambda}_n \Vert^2  \frac{1}{n} \sum_{i=1}^n \Vert g(X_i) \Vert^3 \\
     &\leq (1+o_p(1))  O_p \left(\frac{1}{n} \right) o_p(\sqrt{n}) = o_p\left( \frac{1}{\sqrt{n}} \right).
\end{align*}
Therefore
$$  \hat{\lambda}_n = - \Sigma_n^{-1} \mathbb{P}_n g  + o_p(\frac{1}{\sqrt{n}}).$$
We conclude that 
$ \sqrt{n}\hat{\lambda}_n \Rightarrow \mathcal{N}(0,\Sigma^{-1}).$
\end{proof}

\section{Towards a definition of the informed empirical measure } \label{section_def_mesure}
\subsection{Equivalence of projections and the informed empirical measure}

Assume that $Pg = 0$. In the previous section, we have studied these two optimization problems 
\begin{align}
    &arg \min_{Q \in \mathcal{P}^I\left( \mathcal{Z} \right)} KL ( \mathbb{P}_n || Q ), \label{premier problème} \\
    &arg \min_{Q \in \mathcal{P}^I\left( \mathcal{Z} \right)} KL (Q || \mathbb{P}_n ). \label{second problème}
\end{align}
In this section, we assume that the conditions of existence and uniqueness of the solution of  (\ref{premier problème}) and (\ref{second problème}) are satisfied -- see Theorem \ref{theoreme asymptotique premier problème} and \ref{theoreme asymptotique second problème}. So by Theorem \ref{theoreme_premier_opt}, the solution of the first optimization problem is given by \begin{align*}
    q^{(1)}_i = \frac{1}{n} \frac{1}{1 + \langle \hat{\lambda}_n^{(1)}, g(X_i) \rangle}, \ i \in \llbracket 1,n \rrbracket.
\end{align*}
By Theorem \ref{solution second problème}, the solution of the second optimization problem is given by 
\begin{align*}
    & q^{(2)}_i =  \frac{ \exp \left( \langle \hat{\lambda}_n^{(2)}, g(X_i) \rangle \right),  }{S_n(\hat{\lambda}_n^{(2)} ) } ,  \ i \in \llbracket 1,n \rrbracket, \\
    & S_n(\hat{\lambda}_n^{(2)} ) = \sum_{k=1}^n \exp\left(\langle \hat{\lambda}_n^{(2)}, g(X_k) \rangle \right).
\end{align*}
Denote these two probability measures, respectively, 
\begin{align}
    \mathbb{P}_n^{(1)} &= \sum_{i=1}^n q_i^{(1)} \delta_{X_i}, \label{mesure_vraisemblance} \\
    \mathbb{P}_n^{(2)} &= \sum_{i=1}^n q_i^{(2)} \delta_{X_i}.
\end{align}
Since these weights $q^{(1)}$ and $q^{(2)}$ are not explicit and we don't know if the submanifold $\mathcal{P}^I(\mathcal{Z})$ is autoparallel -- see Theorem \ref{theoreme geometrie} --  it is necessary to find an explicit approximation of these solutions. 

\begin{proposition} \label{Approximation}
Assume that $0$ belongs to the convex hull of $\{g(X_1),\cdots,g(X_n) \}$ and that assumption (\ref{hypothèse_utilisée}) is satisfied. Moreover, suppose that  $\Sigma = Var_P g$ is invertible. Denote $\Sigma_n$ the empirical variance of $g$. Then for all  $i \in \llbracket 1,n \rrbracket$ 
\begin{align*}
    q^{(1)}_i &= p_i + \varepsilon_n^{(1)}, \\
    q^{(2)}_i &= p_i + \varepsilon_n^{(2)},
\end{align*}
with $\varepsilon_n^{(1)}$ and $\varepsilon_n^{(2)}$ independent of $i$ and such that 
\begin{align*}
     \varepsilon_n^{(j)} &= o_p\left( \frac{1}{n}\right), \  j \in \llbracket 1,2 \rrbracket, \\
    p_i &= \frac{1}{n} \left( 1 - g(X_i)^T \Sigma_n^{-1} \mathbb{P}_n g + (\mathbb{P}_n g)^T \Sigma_n^{-1} \mathbb{P}_n g \right) , \ i \in \llbracket 1,n \rrbracket. 
\end{align*}
Moreover it holds 
\begin{align*}
    \sum_{i=1}^n p_i &= 1, \\
    \sum_{i=1}^n p_i g(X_i) &= 0.
\end{align*}
\end{proposition}
\begin{proof}
First assume that $j = 1$. Then for all  $i \in  \llbracket 1,n \rrbracket $ 
\begin{align*}
    q^{(1)}_i = \frac{1}{n} \frac{1}{1 + \langle \hat{\lambda}_n^{(1)}, g(X_i) \rangle}
\end{align*}
and, by Theorem \ref{theoreme asymptotique premier problème} we have
\begin{align*}
    &\max_{1 \leq k \leq n} |\langle \hat{\lambda}_n^{(1)}, g(X_k) \rangle | = o_p(1), \\
    &\hat{\lambda}_n^{(1)} = \Sigma_n^{-1} \mathbb{P}_n g + o_p\left( \frac{1}{\sqrt{n}}\right).
\end{align*}
Thus
\begin{align*}
    q^{(1)}_i = \frac{1}{n} \left( 1 -  \langle \hat{\lambda}_n^{(1)}, g(X_i) \rangle + o_p(1) \right) = \frac{1}{n} \left( 1 -  \langle \Sigma_n^{-1} \mathbb{P}_n g , g(X_i) \rangle  - \langle  o_p\left( \frac{1}{\sqrt{n}}\right) , g(X_i) \rangle + o_p(1) \right).
\end{align*}
Since  $\max_{1 \leq k \leq n} \Vert g(X_k) \Vert = o_p(\sqrt{n})$ -- see Owen's lemma in \cite{owen} -- we deduce that
\begin{align*}
    q^{(1)}_i &= \frac{1}{n} \left( 1 -  \langle \Sigma_n^{-1} \mathbb{P}_n g , g(X_i) \rangle   + o_p(1) \right) = \frac{1}{n} \left( 1 -  \langle \Sigma_n^{-1} \mathbb{P}_n g , g(X_i) \rangle \right) + o_p\left( \frac{1}{n}\right).
\end{align*}
Moreover since  $\frac{1}{n}(\mathbb{P}_n g)^T \Sigma_n^{-1} \mathbb{P}_n g = o_p\left( \frac{1}{n}\right) $ we get
\begin{align*}
    q^{(1)}_i = p_i + o_p\left( \frac{1}{n}\right).
\end{align*}
Now assume that $j=2$. Then for all $i \in [|1,n|]$,
\begin{align*}
    & q^{(2)}_i =  \frac{ \exp \left( \langle \hat{\lambda}_n^{(2)}, g(X_i) \rangle \right) }{S_n(\hat{\lambda}_n^{(2)} ) },\\
  \text{with} \    & S_n(\hat{\lambda}_n^{(2)} ) = \sum_{k=1}^n \exp\left(\langle \hat{\lambda}_n^{(2)}, g(X_k) \rangle \right).
\end{align*}
Hence Theorem \ref{theoreme asymptotique second problème} implies
\begin{align*}
    &\max_{1 \leq k \leq n} |\langle \hat{\lambda}_n^{(2)}, g(X_k) \rangle | = o_p(1),\\
    &\hat{\lambda}_n^{(2)} = - \Sigma_n^{-1} \mathbb{P}_n g + o_p\left( \frac{1}{\sqrt{n}}\right),
\end{align*}
and
\begin{align*}
    q^{(2)}_i &=  \frac{1}{n} \frac{n}{ S_n(\hat{\lambda}_n^{(2)} ) } \left( 1 +  \langle \hat{\lambda}_n^{(2)}, g(X_i) \rangle + o_p(1) \right) = \frac{1}{n} \frac{n}{ S_n(\hat{\lambda}_n^{(2)} ) } \left( 1 -  \langle \Sigma_n^{-1} \mathbb{P}_n g , g(X_i) \rangle  + \langle  o_p\left( \frac{1}{\sqrt{n}}\right) , g(X_i) \rangle + o_p(1) \right)  .
\end{align*}
We deduce that $q^{(2)}_i =  \frac{1}{n} \frac{n}{ S_n(\hat{\lambda}_n^{(2)} ) } \left( 1  -  \langle \Sigma_n^{-1} \mathbb{P}_n g , g(X_i) \rangle + o_p(1) \right)$.

Let us prove that  $\frac{n}{ S_n(\hat{\lambda}_n^{(2)} ) } = 1+o_p(1)$. Since $e^x \geq 1 + x$ for all $x \in \R$ we have 
\begin{align*}
     1 + \langle \hat{\lambda}_n^{(2)}, \mathbb{P}_n g \rangle \leq \frac{ S_n(\hat{\lambda}_n^{(2)} ) }{n}  \leq \exp\left( \max_{1 \leq k \leq n} |\langle \hat{\lambda}_n^{(2)}, g(X_k) \rangle | \right).
\end{align*}
It ensues that $\frac{n}{ S_n(\hat{\lambda}_n^{(2)} ) } = 1+o_p(1)$ and
\begin{align*}
    q^{(2)}_i &= \frac{1}{n} (1+o_p(1)) \left( 1  -  \langle \Sigma_n^{-1} \mathbb{P}_n g , g(X_i) \rangle + o_p(1) \right) =\frac{1}{n} \left( 1  -  \langle \Sigma_n^{-1} \mathbb{P}_n g , g(X_i) \rangle + o_p(1) \right).
\end{align*}
Likewise since $\frac{1}{n}(\mathbb{P}_n g)^T \Sigma_n^{-1} \mathbb{P}_n g = o_p\left( \frac{1}{n}\right) $ it holds $q^{(2)}_i = p_i + o_p\left( \frac{1}{n}\right).$

Finally remark that
\begin{align*}
    \sum_{i=1}^n p_i &= 1 - (\mathbb{P}_n g)^T \Sigma_n^{-1} \mathbb{P}_n g + (\mathbb{P}_n g)^T \Sigma_n^{-1} \mathbb{P}_n g = 1, \\
    \sum_{i=1}^n p_i g(X_i) &= \mathbb{P}_n g - \mathbb{P}_n gg^T \Sigma_n^{-1} \mathbb{P}_n g + \mathbb{P}_n g (\mathbb{P}_n g)^T \Sigma_n^{-1} \mathbb{P}_n g = \mathbb{P}_n g - \mathbb{P}_n g = 0,
\end{align*}
by using $\Sigma_n = \mathbb{P}_n gg^T - \mathbb{P}_n g (\mathbb{P}_n g)^T$. 
\end{proof}

Proposition \ref{Approximation} allows to define the informed empirical measure as follows.
\begin{definition}\label{definition_mesure_informée}
Assume that  $\Sigma = Var_P g$ is invertible. The informed empirical measure is defined to be
\begin{align*}
   \mathbb{P}_n^I :=  \sum_{i=1}^n p_i \delta_{X_i}
\end{align*}
where, for all $i \in  \llbracket 1,n \rrbracket $ 
\begin{align*}
    p_i = \frac{1}{n} \left( 1 - g(X_i)^T \Sigma_n^{-1} \mathbb{P}_n g + (\mathbb{P}_n g)^T \Sigma_n^{-1} \mathbb{P}_n g \right).
\end{align*}
\end{definition}
\begin{remark}
$\mathbb{P}_n^I$ is always defined as soon as  $\Sigma_n = Var_n g$ is invertible. 
\end{remark}
Observe that for any measurable function  $f : \mathcal{X} \to \R$ it holds
\begin{align*}
    \mathbb{P}_n^I f = \mathbb{P}_n f  - cov_n(g,f)^T \Sigma_n^{-1} \mathbb{P}_n g.
\end{align*}
It turns out that this measure coincides with the adaptive estimator of the measure with auxiliary information studied by M. Albertus \cite{albertus} which is a particular case of the general principle of S. Tarima and D. Pavlov \cite{tarima}.

The following corollary states that $\mathbb{P}_n^I$ has the same asymptotic properties as $\mathbb{P}_n^{(1)}$ and  $\mathbb{P}_n^{(2)}$.

\begin{corollary}
For all $j \in  \llbracket 1,2 \rrbracket$ and for any function $f : \mathcal{X} \to \R$ that is integrable with respect to $P$ 
\begin{align*}
    \sqrt{n} (\mathbb{P}_n^{(j)} - P)(f) = \sqrt{n} \left(\mathbb{P}_n^{I} - P\right)(f) + r_n^{(j)}  \mathbb{P}_n f
\end{align*}
where for all  $j \in  \llbracket 1,2 \rrbracket $, $r_n^{(j)} = o_p\left( \frac{1}{\sqrt{n}} \right)$ is independent of  $f$. 
\end{corollary}
\begin{proof}
We apply Proposition \ref{Approximation} and set for all  $j \in  \llbracket 1,2 \rrbracket $ , $r_n^{(j)} = \sqrt{n} \varepsilon_n^{(j)}$. 
\end{proof}
\begin{remark}
Let $\mathcal{F}$ be a class of function. If $\sqrt{n} \left(\mathbb{P}_n^{I} - P\right)$ indexed by $\mathcal{F}$ converges in distribution to a limit process $G_I$ in $l^\infty(\mathcal{F})$ then for all $j \in  \llbracket 1,2 \rrbracket $ the empirical process $\sqrt{n} (\mathbb{P}_n^{(j)} - P)$ converges also in distribution to $G_I$ in $l^\infty(\mathcal{F})$. 
\end{remark}

\subsection{Weights of the informed empirical measure}

Let illustrate the difference of these four measures $\mathbb{P}_n^{(1)}$, $\mathbb{P}_n^{(2)}$, $\mathbb{P}_n^I$ and $\mathbb{P}_n$ by comparing the distribution of weights between them -- see Figure $1$. To this aim we simulate $n=500$ i.i.d. random variables  with distribution $P = \mathcal{N}(0,1)$ and we incorporate the auxiliary information $I$ given by $Pg$ with for all $x \in \R$, $g(x)= (x,x^2)^T$. Observe that the distribution of weights between these three measures $\mathbb{P}_n^{(1)}$, $\mathbb{P}_n^{(2)}$, $\mathbb{P}_n^I$ are very similar.
\begin{figure}[h]
\begin{center}
\includegraphics[scale = 0.5]{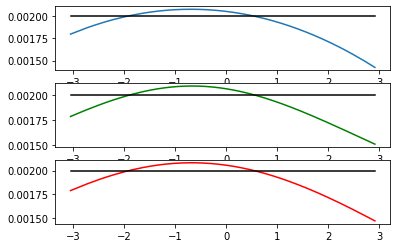} 
\qquad
\includegraphics[scale = 0.5]{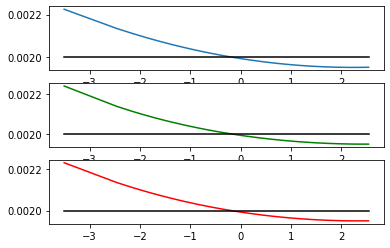}
\caption{Comparison between $\mathbb{P}_n^I$ (blue), $\mathbb{P}_n^{(1)}$ (green), $\mathbb{P}_n^{(2)}$ (red) and $\mathbb{P}_n $ (black).}
\label{Figure}
\end{center}
\end{figure}

The following proposition states that under a moment condition, with probability one $\mathbb{P}_n^I$ is a probability measure for $n$ sufficiently large. 
\begin{proposition} \label{mesure_proba}
Assume that there exists  $\varepsilon > 0$ such that  $P \Vert g\Vert^{4+\varepsilon} < +\infty$. Then almost surely for $n$ sufficiently large 
\begin{align*}
    \mathbb{P}_n^I = \sum_{i=1}^n p_i \delta_{X_i}
\end{align*}
is a probability measure. 
\end{proposition}
\begin{proof}
By Proposition \ref{Approximation}, it is enough to prove that almost surely for all $n$ sufficiently large it holds $\min_{1 \leq i \leq n} p_i > 0$. For that, we need the following technical lemma that extends Owen's lemma \ref{lemme_Owen}. 
\begin{lemma}
Let $(Y_n)_{n \in \N^*}$ be i.i.d. positive random variables and $Z_n = \max_{1 \leq i \leq n} Y_i$. If $\ \mathbb{E}Y^s < + \infty$ with $s > 0$ then 
\begin{align*}
    Z_n &= o_{a.s.}(n^{\frac{1}{s}}), \\
    \frac{1}{n} \sum_{i=1}^n Y_i^{s+1} &= o_{a.s.}(n^{\frac{1}{s}}).
\end{align*}
\end{lemma}
\begin{proof}
Since $\mathbb{E}Y^s < +\infty$, we have  $$\sum_{n \geq 1} \mathbb{P}(Y_n > n^{1/s})=\sum_{n \geq 1} \mathbb{P}(Y_1^s > n) < + \infty$$

and, by Borel-Cantelli lemma 
$Y_n > n^{\frac{1}{s}}$ is a.s. satisfied for a finite set in $\N^*$. Therefore, for all $A > 0$, $Z_n > An^{\frac{1}{s}} $ is a.s. satisfied for a finite set in $\N^*$.

Since $\N^*$ is countable, we then a.s. have, for all $m \in \N^*$,
$$ 0 \leq  \liminf_n{\frac{Z_n}{n^{\frac{1}{s}}}} \leq \limsup_n{\frac{Z_n}{n^{\frac{1}{s}}}} \leq \frac{1}{m}$$
hence $Z_n = o_{a.s.}(n^{\frac{1}{s}})$. 

Concerning the second assertion, observe that 
$$ 0 \leq \frac{1}{n} \sum_{i=1}^n Y_i^{s+1} \leq \frac{Z_n}{n} \sum_{i=1}^n Y_i^s = o_{a.s.}(n^{\frac{1}{s}}) $$
by the strong law of large numbers.  
\end{proof}
Since $P \Vert g \Vert^{4+\varepsilon} < + \infty$ by assumption, the previous lemma yields
\begin{align*}
    \frac{\max_{1 \leq i \leq n} \Vert g(X_i) \Vert}{n^{\frac{1}{4+\varepsilon}}}  \xrightarrow[n \to \infty]{a.s.} 0.
\end{align*}
Remark that
\begin{align*}
    \min_{1 \leq i \leq n} np_i &= 1 - \max_{1 \leq i \leq n} g(X_i)^T \Sigma_n^{-1} \mathbb{P}_n g  + (\mathbb{P}_n g)^T \Sigma_n^{-1} \mathbb{P}_n g, \\
    |\max_{1 \leq i \leq n} g(X_i)^T \Sigma_n^{-1} \mathbb{P}_n g | &\leq \max_{1 \leq i \leq n} | g(X_i)^T \Sigma_n^{-1} \mathbb{P}_n g| \\
    & \leq \max_{1 \leq i \leq n} \Vert g(X_i) \Vert \Vert \Sigma_n^{-1} \Vert \Vert \mathbb{P}_n g \Vert.
\end{align*}
It suffices to prove that
\begin{align*}
 n^{\frac{1}{4+\varepsilon}} \Vert \mathbb{P}_n g \Vert \xrightarrow[n \to \infty]{a.s.} 0.
\end{align*}
For that, we need the following Theorem of \textit{Wellner} and \textit{Van der Vaart} in chapter 2.5 \cite{wellner}.
\begin{lemma}
Let $\mathcal{F}$ be a  $P$-measurable class with envelope function $F$. Then for all $p \geq 1$, there exists $c_p>0$  such that,
\begin{align*}
    \Vert \Vert \alpha_n \Vert^*_{\mathcal{F}} \Vert_{L^p(P)} \leq c_p J(1,\mathcal{F}) \Vert F \Vert_{L^{2 \vee p}(P)}
\end{align*}
where $$ J(1,\mathcal{F}) = \sup_{Q} \ \int_{0}^1 \sqrt{1+\log N \left( \epsilon \Vert F \Vert_{L^2(Q)},\mathcal{F},L^2(Q) \right) } \ d\epsilon.$$
and $N \left( \epsilon \Vert F \Vert_{L^2(Q)},\mathcal{F},L^2(Q) \right)$ is the minimal numbers of balls of radius $\epsilon \Vert F \Vert_{Q,2}$ needed to  cover $\mathcal{F}$ in $L^2(Q)$. Here the supremum is taken over all finitely discrete probability measure on $\left( \mathcal{X}, \mathcal{A} \right)$. 
\end{lemma}
Set for all $j \in  \llbracket 1,m \rrbracket $, $\mathcal{F}_j = \{g_j\}$ and denote $s=4+\varepsilon$. Remark that  $J(1,\mathcal{F}_j) = 1$. Then 
\begin{align*}
    \mathbb{P}(n^{\frac{1}{s}} \Vert \mathbb{P}_n g \Vert > \epsilon) &\leq \sum_{j=1}^m \mathbb{P}(n^{\frac{1}{s}} \mid \mathbb{P}_n g_j \mid > \epsilon) 
    = \sum_{j=1}^m \mathbb{P}\left( \mid \alpha_n (g_j) \mid^s > \left(\frac{\sqrt{n}}{n^{\frac{1}{s}}}\epsilon \right)^s \right).
\end{align*}
So 
\begin{align*}
       \mathbb{P}(n^{\frac{1}{s}} \Vert \mathbb{P}_n g \Vert > \epsilon) \leq \sum_{j=1}^m \frac{\mathbb{E}(\mid \alpha_n(g_j) \mid^s)}{n^{\alpha(1/2 - 1/s)} \epsilon^s} \leq c_s \sum_{j=1}^m \frac{\Vert g_j \Vert_{L^s(P)}^s}{n^{s\frac{s-2}{2s}} \epsilon^s}.
\end{align*}
Since $s > 4$, we get  $s\frac{s-2}{2s} > 1$ and,  by Borel-Cantelli's lemma, $n^{\frac{1}{s}} \Vert \mathbb{P}_n g \Vert \xrightarrow[n \to \infty]{a.s.} 0$. We conclude that 
\begin{align*}
 n^{\frac{1}{4+\varepsilon}} \Vert \mathbb{P}_n g \Vert \xrightarrow[n \to \infty]{a.s.} 0.
\end{align*}

\end{proof}


The weights  $(p_i)_{1 \leq i \leq n}$ are given for all  $i \in  \llbracket 1,n \rrbracket $ 
\begin{align*}
    p_i = \frac{1}{n} \left( 1 - g(X_i)^T \Sigma_{n}^{-1} \mathbb{P}_n g + (\mathbb{P}_n g)^T \Sigma_{n}^{-1} \mathbb{P}_n g \right).
\end{align*}
Set $A_n = - \frac{1}{n} \Sigma_{n}^{-1} \mathbb{P}_n g $ and $B_n = \frac{1}{n} \left( 1 + \langle \Sigma_{n}^{-1} \mathbb{P}_n g , \mathbb{P}_n g  \rangle \right) $. Define for all $(x,y) \in \mathcal{X} \times \R^m$  
\begin{align*}
    \psi_n(y) &= \langle A_n ,y \rangle + B_n , \\
    \varphi_n(x) &= \langle A_n ,g(x) \rangle + B_n = \psi_n \circ g (x).
\end{align*}
Observe that for all $i \in  \llbracket 1,n \rrbracket $, $p_i = \varphi_n(X_i)$ and $\varphi_n$ is an affine transformation of $g$. Suppose that $m=1$ and $g$ is monotone. The weights associated to the ordered sample are also ordered. Illustrate this by a simulation. We generate $n = 500$ random variables i.i.d.  with distribution $P = \mathcal{N}(0,1)$. In the left hand (resp. right hand) graph, we plot $\varphi_n$ with $I$ given by $g(x) = x$ (resp. $g(x) = x^2$) for all $x \in \R$.

\begin{figure}[h]
\begin{center}
\includegraphics[scale = 0.5]{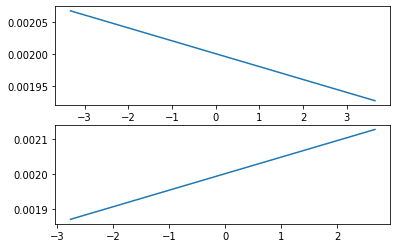}
\qquad
\includegraphics[scale = 0.5]{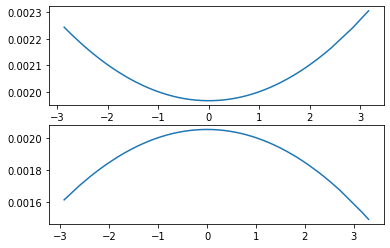}
\caption{The sign of $A_n$ changes randomly.}
\label{Figure}
\end{center}
\end{figure}

\section{Asymptotic results and concentration} \label{section_resultat_asymptotique}
\subsection{ $P$-Glivenko-Cantelli and $P$-Donsker properties under minimal assumptions}

Remind that $\mathbb{P}_n^I$ is the informed empirical measure of Definition \ref{definition_mesure_informée} and write $\alpha_n^I = \sqrt{n} \left( \mathbb{P}_n^I - P\right)$ the informed empirical process. In this section we derive asymptotic results for $\mathbb{P}_n^I$ under minimal assumptions.

Given a class of functions $\mathcal{F}$, if $\Vert \mathbb{P}_n^I - P \Vert_\mathcal{F} = \sup_{f \in \mathcal{F}} |\mathbb{P}_n^I f - Pf |$ is not measurable, its minimal measurable majorant $\Vert \mathbb{P}_n^I - P \Vert_\mathcal{F}^*$ is used -- as well as the outer probability $\mathbb{P}^*$, see Chapter $1.2$ in \cite{wellner}. Denote $\Vert \cdot \Vert$ the euclidean norm on $\R^m$. Let $A \in \mathcal{M}_m \left( \R \right)$ be a matrix, we denote $\Vert A \Vert$ the operator norm with respect to the euclidean norm. 

The following theorem states that a $P$-Glivenko Cantelli class for $\mathbb{P}_n$ is also a $P$-Glivenko Cantelli for $\mathbb{P}_n^I$ as soon as the envelope function $F \in L^2(P)$.

\begin{theorem} \label{resultat_glivenko_cantelli}
Assume that $\Sigma = Var_P g$ is invertible. Let $\mathcal{F}$ be a $P$-Glivenko-Cantelli class with measurable envelope function $F \in L^2(P)$. Then we have  
$$ \Vert \mathbb{P}_n^I - P \Vert_\mathcal{F}^* \xrightarrow[n \to \infty]{a.s.} 0.$$
\end{theorem}
\begin{proof}
Let $f \in \mathcal{F}$. Remark that  
\begin{align*}
     \Vert cov_n(g,f)^T \Sigma_n^{-1} \Vert \leq \sqrt{m} \sqrt{Var_n f} \Vert \Sigma_n^{-1} \Vert \max_{1 \leq i \leq m} \sqrt{Var_n g_i} \leq \sqrt{m} \Vert F \Vert_{L^2(\mathbb{P}_n)} \Vert \Sigma_n^{-1} \Vert \max_{1 \leq i \leq m} \sqrt{Var_n g_i}. 
\end{align*}
Hence
\begin{align*}
    \Vert \mathbb{P}_n^I - P \Vert_\mathcal{F}^* \leq \Vert \mathbb{P}_n - P \Vert_\mathcal{F}^* + \Vert \mathbb{P}_n g \Vert \sup_{f \in \mathcal{F}} \Vert cov_n(g,f)^T \Sigma_n^{-1} \Vert \xrightarrow[n \to \infty]{a.s.} 0.
\end{align*}
\end{proof}
We have a similar result for the $P$-Donsker classes. 
\begin{theorem} \label{resultat_Donsker}
Assume that $\Sigma = Var_P g$ is invertible. Let $\mathcal{F}$ be a $P$-Donsker class with measurable envelope function $F \in L^2(P)$. Then 
$$ \alpha_n^I \Rightarrow G_I \ \ \text{in} \ \ l^\infty(\mathcal{F})$$ 
where $G$ is a $P$-Brownian bridge having almost surely continuous sample paths with respect to the semimetric $\rho^2(h_1,h_2)= Var_P \left( h_1 - h_2\right) $ for  $h_1,h_2 \in L^2(P)$ and, for all $f \in \mathcal{F}$
\begin{align}
    G_I(f) = G(f) - cov_P(g,f)^T \Sigma^{-1} G(g) \label{pont_brownien_informé}
\end{align}
with $G(g) = (G(g_1), \cdots , G(g_m))^T$. 

Moreover for all  $f \in \mathcal{F}$, $Var(G_I(f)) = Var_P f - cov_P(g,f)^T \Sigma^{-1} cov_P(g,f) $ and
\begin{align*}
    Var(G_I(f)) \leq Var(G(f)).
\end{align*}

\end{theorem}
\begin{remark}
Observe that $G_I$ is a mean-zero Gaussian process such that for all $f \in \mathcal{F}$, $Var(G_I(f)) < Var(G(f))$ provided that $cov_P(g,f) \neq 0$. In other words, any function linearly correlated to $g$ benefits of the information $I$.
\end{remark}
\begin{proof}
First remark that 
\begin{align*}
    \alpha_n^I(f) = \alpha_n(f) - cov_P(g,f)^T \Sigma^{-1} \alpha_n(g) +o_p(1).
\end{align*}
Indeed 
\begin{align*}
    (cov_n(g,f)^T \Sigma_n^{-1} - cov(g,f)^T \Sigma^{-1}) \alpha_n(g) &= (cov_n(g,f)^T  - cov_P(g,f)^T)\Sigma_n^{-1} \alpha_n(g) \\ &+ cov_P(g,f)^T(\Sigma_n^{-1} - \Sigma^{-1}) \alpha_n(g).
\end{align*}
The second term tends to $0$ in probability, 
\begin{align*}
    |cov_P(g,f)^T(\Sigma_n^{-1} - \Sigma^{-1})\alpha_n(g)| &\leq \sqrt{m} \sqrt{Var_P f} \max_{1 \leq i \leq m} \sqrt{Var_P g_i}  \Vert \Sigma_n^{-1} - \Sigma^{-1} \Vert \Vert \alpha_n(g) \Vert \\
    &\leq \sqrt{m} \max_{1 \leq i \leq m} \sqrt{Var_P g_i} \Vert F \Vert_{L^2(P)}  \Vert \Sigma_n^{-1} - \Sigma^{-1} \Vert \Vert \alpha_n(g) \Vert =o_p(1).
\end{align*}
For the first term observe that
\begin{align*}
    |(cov_n(g,f)^T  - cov_P(g,f)^T)\Sigma_n^{-1} \alpha_n(g) | &\leq \Vert \mathbb{P}_n fg - Pfg - \mathbb{P}_n f \mathbb{P}_n g \Vert \Vert \Sigma_n^{-1} \Vert \Vert  \alpha_n(g) \Vert \\
    &\leq \sqrt{m}  \max_{1 \leq j \leq m} \Vert \mathbb{P}_n - P \Vert_{\mathcal{F}_j}^* \Vert \Sigma_n^{-1} \Vert \Vert  \alpha_n(g) \Vert \\ &+ |\mathbb{P}_n F | \Vert \mathbb{P}_n g \Vert \Vert \Sigma_n^{-1} \Vert \Vert  \alpha_n(g) \Vert
\end{align*}
where for all $j \in \llbracket 1,m \rrbracket $, $\mathcal{F}_j = \left \{ g_j f, f\in \mathcal{F} \right \}$. 

Since $g$ is integrable with respect to  $P$ then for all $j \in \llbracket 1,m \rrbracket$,  $\{ g_j \}$ is $P$-Glivenko-Cantelli. Moreover $\mathcal{F}$ is $P$-Glivenko-Cantelli since $\mathcal{F}$ is $P$-Donsker. Set for all $(x,y) \in \R^2$, $\phi(x,y) = xy$. Remark that $\mathcal{F}_j = \phi(\{g_j\},\mathcal{F})$. Since $\phi$ is continuous, we deduce that $j \in \llbracket 1,m \rrbracket$, $\mathcal{F}_j$ is a $P$-Glivenko-Cantelli Class. So $$ \max_{1 \leq j \leq m} \Vert \mathbb{P}_n - P \Vert_{\mathcal{F}_j}^* =o_{a.s.}(1). $$
Thus  
$$ \sup_{f \in \mathcal{F}} |(cov_n(g,f)^T \Sigma_n^{-1} - cov_P(g,f)^T \Sigma^{-1}) \alpha_n(g) | =o_p(1).$$
Denote for all $f \in \mathcal{F}$  $$ W_n(f) = \alpha_n(f) - cov_P(g,f)^T \Sigma^{-1} \alpha_n(g).$$
It remains to prove that $W_n$ converges in distribution.
Remark that by the central limit theorem (CLT) 
\begin{align*}
   \begin{pmatrix} \alpha_n(f) \\   \alpha_n(g) \end{pmatrix} \Rightarrow \begin{pmatrix} G(f) \\ G(g) \end{pmatrix} , \ f \in \mathcal{F}
\end{align*}
where $G$ is a $P$-Brownian bridge. Define, for $(f,a,b) \in \mathcal{F} \times \R^2$, $\psi_f(a,b)= (a,cov_P(g,f)^T \Sigma^{-1} b)$. Since the map $\psi_f$ is continuous on $\R^2$, by the continuous mapping theorem we have 
\begin{align*}
   \widetilde{\mathbb{Y}}_n(f) =  \ \begin{pmatrix} \alpha_n(f) \\  cov_P(g,f)^T \Sigma^{-1} \alpha_n(g) \end{pmatrix} \Rightarrow \widetilde{\mathbb{Y}}(f) = \begin{pmatrix} G(f) \\ cov_P(g,f)^T \Sigma^{-1} G(g) \end{pmatrix}, \ f \in \mathcal{F}.
\end{align*}
Likewise for all $ f= (f_1, \cdots, f_k) \in \mathcal{F}^k$, $\left( \widetilde{\mathbb{Y}}_n(f_1), \cdots , \widetilde{\mathbb{Y}}_n(f_k) \right)$ converges in distribution to $\left( \widetilde{\mathbb{Y}}(f_1), \cdots , \widetilde{\mathbb{Y}}(f_k) \right)$. Indeed it suffices to apply the CLT and to consider the following continuous map 
$$ \tilde{\psi}_f(a_1,...,a_{k+1}) = \left(\psi_{f_1}(a_1,a_{k+1}),..., \psi_{f_k}(a_k,a_{k+1})\right), \ (a_1,...,a_{k+1}) \in \R^{k+1}.$$
Applying once again the continuous mapping theorem we deduce 
\begin{align*}
    \left( W_n(f_1), \cdots , W_n(f_k) \right)^T \Rightarrow \left(  G(f_1) - cov_P(g,f_1)^T \Sigma^{-1} G(g), \cdots ,  G(f_k) - cov_P(g,f_k)^T \Sigma^{-1} G(g) \right)^T.
\end{align*}
Let us prove that $W_n \Rightarrow G_I$ in $l^\infty(\mathcal{F})$ with $G_I$ defined at (\ref{pont_brownien_informé}). To this aim, we need the following theorem -- see chapter 1.5 in \cite{wellner}.
\begin{lemma}
Let $X_n : \Omega_n \rightarrow l^\infty(T)$ a sequence of maps. Then these two assertions are equivalents 
\begin{description}
\item[•] We have  
\begin{enumerate}
    \item For all $(t_1,...,t_k) \in T^k$, $(X_n(t_1),...,X_n(t_k))$ converges weakly to a $\mathbb{R}^{k}$ valued random vector, for all  $k \in \N^*$,
    \item  There exists a semimetric $\rho$  such that $(T,\rho)$ is totally bounded and  for all  $ \epsilon >0 $   $$ \lim_{\delta \to 0} \limsup_n \  \mathbb{P}^{*} \left(\sup_{\rho(s,t)<\delta}|X_n(s)-X_n(t)|>\epsilon \right) = 0.$$
\end{enumerate}
\item[•] There exists $X : \Omega \rightarrow l^\infty(T)$ a measurable and tight process such that
$$ X_n \Rightarrow X \ \ \text{in} \ l^\infty(T). $$ 
\end{description}
\end{lemma}

First verify that a.s  $W_n \in l^\infty(\mathcal{F})$ 
\begin{align*}
    \Vert W_n \Vert_{\mathcal{F}} &\leq 2 \max\left(\Vert \alpha_n \Vert_\mathcal{F}, \sup_{f \in \mathcal{F}} |(Pfg)^T \Sigma^{-1} \alpha_n(g)  |\right) \\ &\leq 2 \max\left(\Vert \alpha_n \Vert_\mathcal{F} ,  \sqrt{m} \max_{1 \leq i \leq m} \sqrt{Var_P g_i} \Vert F \Vert_{L^2(P)}  \Vert  \Sigma^{-1} \Vert \Vert \alpha_n(g) \Vert \right) < + \infty.
\end{align*}

In order to check the second point let introduce the usual semimetric  in $L^2(P)$ defined by  $\rho^2(f_1,f_2) = Var_P \left( f_1 - f_2 \right) $ for all $f_1, f_2 \in L^2(P)$. Since $\mathcal{F}$ is $P$-Donsker then $(\mathcal{F}, \rho)$ is totally bounded. For $\varepsilon > 0$  and $\delta > 0$, we have   
\begin{align*}
     & \limsup_n \ \mathbb{P}^{*}\left(\sup_{\rho(f_1,f_2)<\delta} \displaystyle \mid  W_n(f_1)- W_n(f_2) \mid >  2\epsilon\right) \\ &\leq \limsup_n \ \mathbb{P}^{*}\left(  \sup_{\rho(f_1,f_2)<\delta} \max\left(| \alpha_n(f_1) - \alpha_n(f_2)|, |(P(f_1-f_2)g)^T \Sigma^{-1} \alpha_n(g)) |\right) >\epsilon\right) \\
     &\leq  \limsup_n \ \mathbb{P}^{*}\left(\sup_{\rho(f_1,f_2)<\delta} | \alpha_n(f_1) - \alpha_n(f_2)| >\epsilon\right) \\ &+   \limsup_n \ \mathbb{P}^{*}\left(\sup_{\rho(f_1,f_2)<\delta}  \sqrt{m} \max_{1 \leq i \leq m} \sqrt{Var_P g_i} \   \rho(f_1,f_2)   \Vert \Sigma^{-1} \Vert \Vert \alpha_n(g) \Vert >\epsilon \right) \\
     &\leq \limsup_n \ \mathbb{P}^{*}\left(\sup_{\rho(f_1,f_2)<\delta} | \alpha_n(f_1) - \alpha_n(f_2)| >\epsilon\right) \\ &+   \limsup_n \ \mathbb{P}^{*}\left(  \delta \sqrt{m} \max_{1 \leq i \leq m} \sqrt{Var_P g_i}   \Vert \Sigma^{-1} \Vert \Vert \alpha_n(g)  \Vert >\epsilon\right).
\end{align*}
Since $\mathcal{F}$ is $P$-Donsker, we have
\begin{align*}
    \lim_{\delta \to 0} \limsup_n \ \mathbb{P}^{*}\left(\sup_{\rho(f_1,f_2)<\delta} | \alpha_n(f_1) - \alpha_n(f_2)| >\epsilon \right) = 0.
\end{align*}
By using the fact that
\begin{align*}
      \Vert \alpha_n(g) \Vert \Rightarrow \Vert G(g) \Vert,
\end{align*}
we have by the porte-manteau lemma  
\begin{align*}
    \limsup_n \ \mathbb{P}^{*}\left( \delta \sqrt{m} \max_{1 \leq i \leq m} \sqrt{Var_P g_i}   \Vert \Sigma^{-1} \Vert \Vert \alpha_n(g) >\epsilon \right) &\leq \limsup_n \ \mathbb{P}^{*} \left( \delta \sqrt{m} \max_{1 \leq i \leq m} \sqrt{Var_P g_i}   \Vert \Sigma^{-1} \Vert \Vert \alpha_n(g) \geq \epsilon \right)\\
    &\leq \mathbb{P}^{*} \left(  \delta \sqrt{m} \max_{1 \leq i \leq m} \sqrt{Var_P g_i}   \Vert \Sigma^{-1} \Vert \Vert G(g) \Vert \geq \epsilon \right).
\end{align*}
Since  $ \delta \sqrt{m} \max_{1 \leq i \leq m} \sqrt{Var_P g_i}   \Vert \Sigma^{-1} \Vert \Vert \Vert G(g) \Vert = o_p(1)$, we get 
\begin{align*}
    \lim_{\delta \to 0} \limsup_n \ \mathbb{P}^{*} \left(\sup_{\rho(f_1,f_2)<\delta}\displaystyle \mid  W_n(f_1)- W_n(f_2) \mid >  2\epsilon\right) = 0.
\end{align*}
This establishes that $W_n \Rightarrow G_I $ in $l^\infty(\mathcal{F})$. Thus
\begin{align*}
    \alpha_n^I  \Rightarrow G_I \ \ \text{in} \ \ l^\infty(\mathcal{F}).
\end{align*}
Finally we compute the variance of $G_I(f)$ for $f \in \mathcal{F}$,
\begin{align*}
Var\left(G_I(f)\right) &= Var\left(G(f) - cov_P(g,f)^T \Sigma^{-1} G(g) \right) \\
            &= Var\left(G(f) \right) + Var \left( cov_P(g,f)^T \Sigma^{-1} G(g) \right)  - 2 cov\left(G(f) ,  cov_P(g,f)^T \Sigma^{-1} G(g) \right) \\
            &= Var_P f +  cov_P(g,f)^T \Sigma^{-1} \left( Var_P g \right) \Sigma^{-1} cov_P(g,f) - 2 cov_P(g,f)^T \Sigma^{-1} cov_P(g,f) \\
            &= Var_P f - cov_P(g,f)^T \Sigma^{-1} cov_P(g,f).
\end{align*}
Since $\Sigma^{-1} > 0$, we deduce that $cov_P(g,f)^T \Sigma^{-1} cov_P(g,f) \geq 0$. Thus 
$$ Var(G_I(f)) \leq Var(G(f)).$$

\end{proof}
\subsection{Concentration of the informed empirical process} \label{concentration_section}

Next we show that the informed empirical process is more concentrated than the classical empirical process for all $n$ sufficiently large. Moreover, we prove that the supremum of limit process $G_I $ of Theorem \ref{resultat_Donsker} on a $P$-Donsker class is more concentrated than the supremum of $G$.
By Theorem \ref{resultat_Donsker}, we have for all $f \in L^2(P)$ 
\begin{align*}
    Var \left( G_I(f) \right) \leq Var \left( G(f) \right).
\end{align*}

A first consequence is that for all $f \in L^2(P)$, the informed empirical process $\alpha_n^I(f)$ is more concentrated than $\alpha_n(f)$.
\begin{proposition}
Assume that $\Sigma = Var_P g$ is invertible and let $f \in L^2(P)$. Then for any $\lambda > 0$

\begin{align*}
    \mathbb{P}\left( |G_I(f) | > \lambda \right) \leq \mathbb{P}\left( |G(f) | > \lambda \right),
\end{align*}
and, if moreover $cov_P(g,f) \neq 0 $ then  
\begin{align*}
    \mathbb{P}\left( |G_I(f) | > \lambda \right) < \mathbb{P}\left( |G(f) | > \lambda \right).
\end{align*}
In addition there exists $N > 0$ such that for all $n \geq N$ it holds
\begin{align*}
    \mathbb{P}\left( |\alpha_{n}^I(f) | > \lambda \right) < \mathbb{P}\left( |\alpha_{n}(f) | > \lambda \right).
\end{align*}
\end{proposition}
\begin{proof}
Let $ \lambda > 0 $. By Theorem \ref{resultat_Donsker} and Central Limit Theorem (CLT) we have 
\begin{align*}
    &\mathbb{P}\left( |\alpha_{n}^I(f) | > \lambda \right) \xrightarrow[n \to \infty]{}  \mathbb{P}\left( |G_I(f) | > \lambda \right) , \\
    &\mathbb{P}\left( |\alpha_{n}(f) | > \lambda \right) \xrightarrow[n \to \infty]{}  \mathbb{P}\left( |G(f) | > \lambda \right).
\end{align*}
Denote $\sigma_1 := \sqrt{Var G_I(f)} <\sqrt{Var G(f)} =: \sigma_2  $. Observe that  
\begin{align*}
    \mathbb{P}\left( |G_I(f) | > \lambda \right) = \mathbb{P}\left( | \mathcal{N}(0,1) | > \frac{\lambda}{\sigma_1} \right) < \mathbb{P}\left( | \mathcal{N}(0,1) | > \frac{\lambda}{\sigma_2} \right) = \mathbb{P}\left( |G(f) | > \lambda \right).
\end{align*}
So there exists  $N > 0$ such that for all  $ n \geq N$  
\begin{align*}
    \mathbb{P}\left( |\alpha_{n}^I(f) | > \lambda \right) < \mathbb{P}\left( |\alpha_{n}(f) | > \lambda \right).
\end{align*}
\end{proof}

The next step is to extend this result to a $P$-Donsker class $\mathcal{F}$. Recall that $\mathcal{F}$ is a pointwise separable class -- see \cite{wellner} for more details -- if there exists a countable subset $\mathcal{G} \subset \mathcal{F}$ such that for each $n \in \N^*$ there is a $P^n$-null set $N_n \subset \mathcal{X}^n$ such that for all $(x_1, \cdots, x_n) \notin N_n$ and $f \in \mathcal{F}$, there exists a sequence $(h_k)_{k \in \N^*} \subset \mathcal{G}$ such that $h_k \xrightarrow[k \to \infty]{} f$ in $L^2(P)$ and $(h_k(x_1), \cdots, h_k(x_n)) \xrightarrow[k \to \infty]{} (f(x_1), \cdots, f(x_n))$.
\begin{proposition}\label{thmsup}
Assume that $\Sigma = Var_P g$ is invertible and let $\mathcal{F} $ be a $P$-Donsker class. Suppose that $\mathcal{F}$ is pointwise separable. Then for all $\lambda > 0$  
\begin{align*}
    \mathbb{P}\left( \sup_{f \in \mathcal{F}} |G_I(f)| > \lambda \right) \leq \mathbb{P}\left( \sup_{f \in \mathcal{F}} |G(f)| > \lambda \right).
\end{align*}
Moreover if there exists $\lambda > 0$ such that  $\mathbb{P}\left( \sup_{f \in \mathcal{F}} |G_I(f)| > \lambda \right) < \mathbb{P}\left( \sup_{f \in \mathcal{F}} |G(f)| > \lambda \right)$ then there exists $N > 0$ such that for all $n \geq N$ 
\begin{align*}
    \mathbb{P}\left( \sup_{f \in \mathcal{F}} |\alpha_{n,I}(f)| > \lambda \right) < \mathbb{P}\left( \sup_{f \in \mathcal{F}} |\alpha_n(f)| > \lambda \right).
\end{align*}
\end{proposition} 

\begin{proof}
To prove this proposition, we need the following result. 
\begin{lemma}\textit{(Slepian, Fernique, Marcus, Shepp)} \label{lemme_gaussien}

Let $X$ and $Y$ be separable, mean-zero Gaussian processes indexed by a common index set $T$ such that
\begin{align*}
    \mathbb{E}\left(X_s - X_t \right)^2 \leq \mathbb{E}\left(Y_s - Y_t \right)^2 \ \ \text{for all} \ s,t \in T.
\end{align*}

Then for all $\lambda > 0$  
\begin{align*}
    \mathbb{P}\left( \sup_{t \in T} X_t > \lambda \right) \leq \mathbb{P}\left( \sup_{t \in T} Y_t > \lambda \right).
\end{align*}
\end{lemma}
Notice that $G$ and $G_I$ are two mean-zero Gaussian processes and we can take a separable version of these processes. Set $\widetilde{\mathcal{F}} = \mathcal{F} \cup(-\mathcal{F})$ and observe that 
\begin{align*}
    \sup_{f \in \mathcal{F}} |G_I(f)|  &= \sup_{f \in \mathcal{F}} \max \left( G_I(f), - G_I(f) \right) = \sup_{f \in \mathcal{F}} \max \left( G_I(f),  G_I(-f) \right) = \sup_{f \in \widetilde{\mathcal{F}}} G_I(f).
\end{align*}
Similarly $\sup_{f \in \mathcal{F}} |G(f)| =\sup_{f \in \widetilde{\mathcal{F}}} G(f)$. Remark that for all $f,h \in \widetilde{\mathcal{F}}$ 
\begin{align*}
    \mathbb{E}\left(G_I(f) - G_I(h)) \right)^2 &= \mathbb{E}\left(G_I(f - h)) \right)^2  = Var\left(G_I(f - h)\right) \leq Var\left(G(f - h)\right) = \mathbb{E}\left(G(f) - G(h)) \right)^2.
\end{align*}
By Lemma \ref{lemme_gaussien} 
\begin{align*}
     \mathbb{P}\left( \sup_{f \in \mathcal{F}} |G_I(f)| > \lambda \right) =  \mathbb{P}\left( \sup_{f \in \widetilde{\mathcal{F}}} G_I(f) > \lambda \right) \leq \mathbb{P}\left( \sup_{f \in \widetilde{\mathcal{F}}} G(f) > \lambda \right) = \mathbb{P}\left( \sup_{f \in \mathcal{F}} |G(f)| > \lambda \right).
\end{align*}
Assume that there exists $\lambda > 0$ such that  
\begin{align*}
    \mathbb{P}\left( \sup_{f \in \widetilde{\mathcal{F}}} G_I(f) > \lambda \right) < \mathbb{P}\left( \sup_{f \in \widetilde{\mathcal{F}}} G(f) > \lambda \right).
\end{align*}
Observe that the map $X \mapsto \Vert X \Vert_{\mathcal{F}} = \sup_{f \in \mathcal{F}} |X(f)|$ for every $X \in l^\infty(\mathcal{F})$ is continuous. Since $\mathcal{F}$ is pointwise separable, there is no problem with measurability. By Theorem \ref{resultat_Donsker}, we deduce that there exists $N > 0$ such that for all $ n \geq N$  
\begin{align*}
    \mathbb{P}\left( \Vert \alpha_{n,I} \Vert_{\mathcal{F}} > \lambda \right) = \mathbb{P}\left( \sup_{f \in \mathcal{F}} | \alpha_{n,I}(f)| > \lambda \right) < \mathbb{P}\left( \sup_{f \in \mathcal{F}} | \alpha_{n}(f)| > \lambda \right) = \mathbb{P}\left( \Vert \alpha_{n} \Vert_{\mathcal{F}}> \lambda \right).
\end{align*}

\end{proof}

As a corollary we immediately get that the quantiles of $\sup_{f \in \mathcal{F}} |G(f)|$ and $\sup_{f \in \mathcal{F}} |G_I(f)|$ are also ordered. Denote  $F_1$ (resp. $F_2$) the cumulative distribution function of $\sup_{f \in \mathcal{F}} |G(f)|$ (resp. $\sup_{f \in \mathcal{F}} |G_I(f)|$) and, for $ i \in  \llbracket 1,2 \rrbracket $,  $F_i^{-1}(\alpha) =\inf \left\{ t \in \R_+, \ F_i(t) \geq \alpha  \right\}$. 
\begin{corollary}
Assume that $\Sigma = Var_P g$ is invertible and let $\mathcal{F}$ be a $P$-Donsker class. Suppose that $\mathcal{F}$ is pointwise separable. Then, for all $\alpha \in [0,1[$ it holds $ F_2^{-1}(\alpha) \leq F_1^{-1}(\alpha)$.
\end{corollary}

This corollary has many interesting applications. For instance, it can be used in order to improve Kolmogorov-Smirnov test -- see \cite{albertus} at the page $34$ for an auxiliary information given by a partition of $\mathcal{X}$.

\section{Informed empirical quantiles} \label{section_quantile}
As an illustration let consider the informed estimator of a single quantile built from $\mathbb{P}_n^I$ in the case $P$ is a real probability measure. Standard empirical process methods could be applied to extend this estimation uniformly on compact sets of $]0,1[$ -- and on $]0,1[$ under additional assumptions on the regularity and rate of decay of the density in tails. We would obtain similar results and same limiting process as in \cite{zhang_quantile} where the quantile process based on the probability measure $\mathbb{P}_n^{(1)}$ of (\ref{mesure_vraisemblance}) is shown to converge in the appropriated topology. Proposition \ref{Approximation} suggests that similar results are also valid for the quantiles derived from the unusual $\mathbb{P}_n^{(2)}$.

Denote $F$ the cumulative distribution function of $P$ and $\mathbb{F}_n$ the empirical distribution function. Assume that $Pg = 0$. 
We can define for all $t \in \R$ the following function  
\begin{align*}
    \mathbb{F}_{n,I}(t) = \mathbb{P}_n^I 1_{\cdot \leq t} = \sum_{i=1}^n p_i 1_{X_i \leq t}.
\end{align*}
This function is called the informed empirical distribution function since $\mathbb{P}_n^I$ is a probability measure almost surely for $n$ sufficiently large -- see Proposition \ref{mesure_proba}. 

Let us first compare $\mathbb{F}_{n,I}$, $\mathbb{F}_n$ and $F$ through a simulation. For that, we generate $n=100$ random variables i.i.d. with distribution $P=\mathcal{N}(0,1)$ and we incorporate the auxiliary information $I$ given by the function $g(x) = (x,x^2)$ defined for all $x \in \R$.  We remark that $\mathbb{F}_{n,I}$ is closer of $F$ than $\mathbb{F}_n$. 

\begin{figure}[h]
\begin{center}
\includegraphics[scale = 0.5]{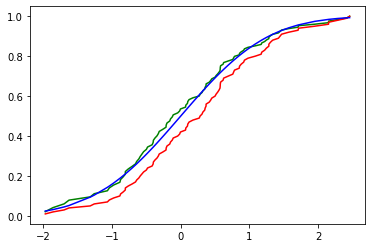} 
\qquad
\includegraphics[scale = 0.5]{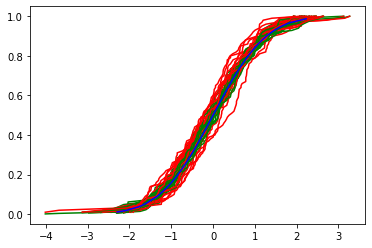}
\caption{ Comparison between $\mathbb{F}_{n,I}$ (green), $\mathbb{F}_{n}$ (red) and $F$ (blue). The number of experiments is $1$ in the left hand graph and $20$ in the right hand graph. }
\label{Figure}
\end{center}
\end{figure}

Given $\alpha \in ]0,1[$, we estimate the $\alpha$-quantile $q_\alpha= F^{-1} (\alpha) = \inf \left\{ t \in \R, \ F(t) \geq \alpha \right\}$ of $P$ by
\begin{align*}
    q_{n,\alpha}^I = \mathbb{F}_{n,I}^{-1} (\alpha) = \inf \left\{ t \in \R, \ \mathbb{F}_{n,I}(t) \geq \alpha \right\}.
\end{align*}
Clearly $q_{n,\alpha}^I$ is well defined since for all $t \in \left]- \infty, \min_{1 \leq i \leq n} X_i \right[, \ F_{n,I}(t) = 0$ and for all  $t \in \left[ \max_{1 \leq i \leq n} X_i, + \infty \right[,  \ F_{n,I}(t) = 1$. The estimator $q_{n,\alpha}^I$ is called the informed empirical $\alpha$-quantile.

In order to asymptotically control $q_{n,\alpha}^I$ let apply the classical delta method.
\begin{lemma} \label{lemme_cadlag}
Let $\mathcal{S} =  D(\R) \cap l^\infty(\R)$ be the set of bounded càdlàg function defined on $\R$ equipped with the uniform norm. Let $\alpha \in ]0,1[$ and $\phi_\alpha : D_\alpha \subset \mathcal{S} \to \R $ be the map defined by  $\phi_\alpha(F) = F^{-1}(\alpha)$ for all $F \in D_\alpha$ with $D_\alpha$ a domain of $\phi_\alpha$ which contains the set of all cumulative distribution function on $\R$ and $\mathbb{F}_{n,I}$ for all $n \in \N^*$, all $(X_n(\omega))_{n \in \N^*}$ and all $\omega\in\Omega$. Let $F$ be a cumulative distribution function.
\begin{itemize}
    \item[$\bullet$] If $F$ is strictly increasing at $F^{-1}(\alpha)$ then  $\phi_\alpha$ is continuous at $F$. 
    \item[$\bullet$] Assume that $F$ is differentiable at $F^{-1}(\alpha)$ and $F^\prime(F^{-1}(\alpha)) = f(F^{-1}(\alpha)) > 0$. Then $\phi_\alpha$ is Hadamard-differentiable in $F$ tangentially to  $$D_0 = \{ h \in \mathcal{S} , \ h \   \text{continue en} \ F^{-1}(\alpha) \}.$$ 
    Moreover
    \begin{align*}
        \phi_\alpha^\prime(h) = - \frac{h(\phi_\alpha(F))}{F^\prime({\phi(F)})} = - \frac{h(F^{-1}(\alpha))}{f(F^{-1}(\alpha))}.
    \end{align*}
\end{itemize} 
\end{lemma}
\begin{remark}
Remark that we can take $D_\alpha =  \left\{ F \in \mathcal{S}, \  \lim_{t \to - \infty} F(t) = 0 , \ \lim_{t \to +\infty} F(t) = 1 \right\}$.
\end{remark}
\begin{proof}
The proof of the second assertion can be found in \cite{wellner} for instance. About the first assertion, let  $\varepsilon > 0$. Since $F$ is strictly increasing at $F^{-1}(\alpha)$ we have
\begin{align*}
    F(F^{-1}(\alpha) - \varepsilon ) < \alpha < F(F^{-1}(\alpha) + \varepsilon ).
\end{align*}
Let $(F_n)_{n \in \N} \subset D_\alpha$ such that $\Vert F_n - F \Vert_\infty \xrightarrow[n \to \infty]{} 0$. Then there exists $N > 0$ such that for all $n \geq N$ 
\begin{align*}
    F_n(F^{-1}(\alpha) - \varepsilon ) < \alpha < F_n(F^{-1}(\alpha) + \varepsilon )
\end{align*}
which implies that
\begin{align*}
    F^{-1}(\alpha) - \varepsilon  \leq F_n^{-1}(\alpha) \leq F^{-1}(\alpha) + \varepsilon.
\end{align*}
Thus $|  F_n^{-1}(\alpha)  -  F^{-1}(\alpha)  | \leq \varepsilon$ and $\phi_\alpha$ is continuous at $F$. 
\end{proof}

The following results shows that $q_{n,\alpha}^I$ has an asymptotic variance strictly less than the uninformed empirical quantile whenever the vector $cov_P(g,1_{]-\infty, F^{-1}(\alpha)]}) \neq 0 $.

\begin{theorem} \label{theoreme_quantile}
Let $F$ be the cumulative distribution function generating the sample and  $q_\alpha$ the $\alpha$-quantile of $F$,  $\alpha \in ]0,1[$.
\begin{enumerate}
    \item[$\bullet$] If $F$ is strictly increasing at  $F^{-1}(\alpha)$ then $$ q_{n,\alpha}^I \xrightarrow[n \to \infty]{a.s.} q_\alpha. $$
    \item[$\bullet$] If $F$ is differentiable at $F^{-1}(\alpha)$ and $F^\prime(F^{-1}(\alpha)) = f(F^{-1}(\alpha)) > 0$ then
    \begin{align*}
        \sqrt{n}(q_{n,\alpha}^I - q_\alpha) \Rightarrow \mathcal{N}\left(0 , \frac{\alpha(1-\alpha) - \tilde{I} }{\left(f(F^{-1}(\alpha)\right)^2} \right)
    \end{align*}
    where $\tilde{I} = cov_P(g,1_{]-\infty, F^{-1}(\alpha)]})^T \Sigma^{-1} cov_P(g,1_{]-\infty, F^{-1}(\alpha)]}) \geq 0$. 
\end{enumerate}
\end{theorem}

\begin{proof}
Remind that the class of function $\mathcal{F} = \{f_s = 1_{]-\infty,s]}, \ s \in \R \}$ is $P$-Donsker and so $P$-Glivenko-Cantelli. By Theorem \ref{resultat_glivenko_cantelli} and \ref{resultat_Donsker}, we have 
\begin{align*}
    \Vert \mathbb{F}_{n,I} - F \Vert_{\infty} &\xrightarrow[n \to \infty]{a.s.} 0, \\ 
    \sqrt{n}(\mathbb{F}_{n,I} - F) &\Rightarrow G_I \ in \ l^\infty(\R),
\end{align*}
where for all $s \in \R$, $G_I(s) := G_I(f_s) = G(f_s) - cov_P(g,f_s)^T \Sigma^{-1} G(g)$. Observe that $G_I \in l^\infty(\R)$. To get the first assertion, apply Lemma \ref{lemme_cadlag} to obtain $$\phi_\alpha(\mathbb{F}_{n,I}) \xrightarrow[n \to \infty]{a.s.} \phi_\alpha(F).$$ 

To derive the second assertion, let prove that $G_I$ is càdlàg on $\R$ and continuous at $F^{-1}(\alpha)$. For that it is enough to prove that these following real-valued maps  are càdlàg on $\R$ and continuous at $F^{-1}(\alpha)$
\begin{align*}
    \psi_1 : s &\mapsto G(f_s), \\
    \psi_2 : s &\mapsto cov_P(g,1_{]-\infty, s]}) = Pg 1_{]-\infty, s]} -  Pg F(s).
\end{align*}
Concerning $\psi_1$, recall that $G$ is a $P$-Brownian bridge having almost surely continuous sample paths with respect to the semimetric $\rho^2(h_1,h_2)= Var_P \left( h_1 - h_2\right) $ for all $h_1,h_2 \in L^2(P)$ -- see Theorem \ref{resultat_Donsker}. The map $\varphi : \left( \R, |\cdot | \right) \to \left( \mathcal{F}, \rho \right)$ defined by $\varphi(s) = f_s$ for all $s \in \R$ is càdlàg and continuous at $F^{-1}(\alpha)$. As a matter of fact, $F$ is continuous at $F^{-1}(\alpha)$ and for all $s,s^\prime$ 
\begin{align*}
    \rho(f_s,f_{s^\prime}) = \sqrt{Var_P \left( f_s - f_{s^\prime}\right) } \leq \Vert f_s - f_{s^\prime} \Vert_{L^2(P)} = \sqrt{F(s) + F(s^\prime) -2 F\left( \min(s,s^\prime) \right) }
\end{align*}
hence $\psi_1$ is càdlàg on $\R$ and continuous at $F^{-1}(\alpha)$. Concerning $\psi_2$, let show that $s \mapsto P\varphi_g(P) 1_{]-\infty, s]}$ is càdlàg. Let $(s_n)_{n \in \N^*} \subset \R$ be a decreasing sequence tending to  $s$ when $n \to +\infty$. Since $g \in L^1(P)$, we deduce by dominated convergence theorem, 
\begin{align*}
   Pg 1_{]-\infty, s_n]} \xrightarrow[n \to \infty]{}  Pg 1_{]-\infty, s]}.
\end{align*}
Likewise, by dominated convergence theorem,
\begin{align*}
    Pg 1_{]-\infty, s^\prime]} \xrightarrow[s^\prime \to s, \ s^\prime < s]{} Pg 1_{]-\infty, s[}.
\end{align*}
Since $F$ is continuous at $F^{-1}(\alpha)$, we deduce that $\psi_2$ is càdlàg and continuous at $F^{-1}(\alpha)$. Finally the functional delta method and Lemma \ref{lemme_cadlag} readily imply 
\begin{align*}
    \sqrt{n}(q_{n,\alpha}^I - q_\alpha) = \sqrt{n}( \phi_\alpha(\mathbb{F}_{n,I})  -  \phi_\alpha(F)) \Rightarrow \phi_\alpha^\prime(G_{I}) = - \frac{G_I(F^{-1}(\alpha))}{f(F^{-1}(\alpha))}.
\end{align*}
Since $F(F^{-1}(\alpha)) = \alpha$, we have $Var_P(1_{]-\infty, F^{-1}(\alpha)]}) = \alpha(1-\alpha) $ . 
\end{proof}

To conclude, besides the asymptotics of Theorem \ref{theoreme_quantile} let show that the information $I$ impacts also small samples. Let use $n=210$ i.i.d. random variables with distribution $P = \mathcal{N}(0,1)$ to estimate sequentially the median of $P$ -- that is $0$. It is assumed that the auxiliary information $I$ is given by $Pg$ with $g(x) = (x,x^2)$ for all $x \in \R$. We draw at figure $4$ $q_{n,\alpha}$ the sequences $q_{n,\alpha}^I$ for $\alpha = \frac{1}{2}$ for every $n \in  \llbracket 2,210 \rrbracket$ -- with the same sample.

\vspace{1cm}

\begin{figure}[h] \label{figure_mediane}
\begin{center}
\includegraphics[scale = 0.8]{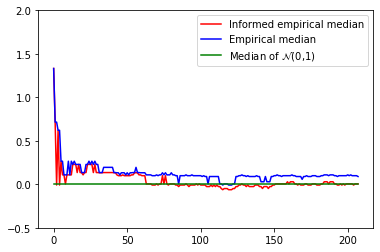} 
\caption{Medians}
\label{Figure}
\end{center}
\end{figure}


\noindent{\bf Acknowledgements}{ We are grateful to P. Berthet for his advises during the preparation of the paper. }

\end{document}